\documentclass[12pt, reqno]{amsart}
\usepackage{amsmath, amsthm, amscd, amsfonts, amssymb, graphicx, color, mathrsfs}
\usepackage[bookmarksnumbered, colorlinks, plainpages]{hyperref}
\usepackage[all]{xy}
\usepackage{slashed}
\usepackage{tipa}
\usepackage{mathabx}
\usepackage{soul}
\usepackage{cancel}
\usepackage{ulem}
\usepackage{inputenc}
\usepackage{booktabs}

\newtheorem{theorem}{Theorem}[section]
\newtheorem{lemma}[theorem]{Lemma}

\newtheorem{proposition}[theorem]{Proposition}
\newtheorem{corollary}[theorem]{Corollary}
\theoremstyle{definition}
\newtheorem{definition}[theorem]{Definition}

\theoremstyle{remark}
\newtheorem{remark}[theorem]{Remark}
\numberwithin{equation}{section}

\newtheorem{assump}[theorem]{Assumption}

\def\ind{{\mathcal I}}
\def\sL{{\star_{L}}}
\def\sLs{\widetilde{\star}_{L}}

\newcounter{quotecount}
\newcommand{\MyQuote}[2]{\bigskip (#2)\hspace*{1cm}
\vspace{1cm}\addtocounter{quotecount}{1}%
     \parbox{12cm}{ #1}}

\begin{document}
\setcounter{page}{1}

\title[ Global Functional calculus, lower/upper bounds  and... ]{  Global Functional calculus, lower/upper bounds  and evolution equations on manifolds with boundary}

\author[D. Cardona]{Duv\'an Cardona}
\address{
  Duv\'an Cardona S\'anchez:
  \endgraf
  Department of Mathematics: Analysis, Logic and Discrete Mathematics
  \endgraf
  Ghent University, Belgium
  \endgraf
  {\it E-mail address} {\rm duvanc306@gmail.com, Duvan.CardonaSanchez@UGent.be}
  }

\author[V. Kumar]{Vishvesh Kumar}
\address{
 Vishvesh Kumar:
  \endgraf
  Department of Mathematics: Analysis, Logic and Discrete Mathematics
  \endgraf
  Ghent University, Belgium
  \endgraf
  {\it E-mail address} {\rm Kumar.Vishvesh@UGent.be}
  }

\author[M. Ruzhansky]{Michael Ruzhansky}
\address{
  Michael Ruzhansky:
  \endgraf
  Department of Mathematics: Analysis, Logic and Discrete Mathematics
  \endgraf
  Ghent University, Belgium
  \endgraf
 and
  \endgraf
  School of Mathematical Sciences
  \endgraf
  Queen Mary University of London
  \endgraf
  United Kingdom
  \endgraf
  {\it E-mail address} {\rm Michael.Ruzhansky@ugent.be}
  }

\author[N. Tokmagambetov]{Niyaz Tokmagambetov}
\address{
 Niyaz Tokmagambetov:
  \endgraf
  Department of Mathematics: Analysis, Logic and Discrete Mathematics
  \endgraf
  Ghent University, Belgium
  \endgraf
  {\it E-mail address} {\rm Niyaz.Tokmagambetov@UGent.be}
  }



\thanks{The  authors were supported by the FWO Odysseus 1 grant G.0H94.18N: Analysis and Partial Differential Equations. The third author was also supported in parts by  EPSRC grant
EP/R003025/1.}

\begin{abstract}
Given a smooth  manifold $M$ (with or without boundary), in this paper we establish a global functional calculus (without the standard assumption that the operators are classical pseudo-differential operators) and  the G\r{a}rding inequality for global pseudo-differential operators associated with boundary value problems. The analysis that we follow is free of local coordinate systems. Applications of the G\r{a}rding inequality to the global solvability for a class of evolution problems are also considered. 
\end{abstract} \maketitle

\tableofcontents
\allowdisplaybreaks

\section{Introduction}
 
Let  $M=\overline{\Omega}$ be a smooth manifold (with or without boundary). This work deals with the $L^2$-theory for the pseudo-differential calculus associated with a boundary-value problem $L_{\Omega},$ determined by a pseudo-differential operator $L$ on $\Omega$ with discrete spectrum, having suitable boundary conditions, in the framework of the non-harmonic analysis developed by the last two authors in \cite{RT2016,RT2017} and \cite{DRT2017}. To be precise, we will formulate a (Dunford-Riesz) global functional calculus for the $(\rho,\delta)$-H\"ormander classes associated to $L_{\Omega}$.  Once established, the functional calculus will be applied to the proof of a  global   G\r{a}rding inequality on $M,$ and in establishing the $L^2$-boundedness for the $(\rho,\delta)$-class of order zero of this calculus. Finally, we will use the $L^2$-theory developed, applying it to the global solvability for a class of  evolution problems on $\Omega,$ associated with (possibly time-dependent) $L$-strongly elliptic pseudo-differential operators on $M,$ (which is the class of elliptic operators in the calculus determined by $L_\Omega$).

In their seminal work \cite{KohnNirenberg1965}, J.  Kohn and L. Nirenberg  introduced a calculus of pseudo-differential operators for the classes $S^{m}_{1,0}(\mathbb{R}^n\times \mathbb{R}^n),$\footnote{which consists of all smooth functions $a$  satisfying $|\partial_{x}^\beta\partial_{\xi}^\alpha a(x,\xi)|=O(1+|\xi|)^{m-|\alpha|},$ when $|\xi|\rightarrow\infty.$} which was applied by  K.  K. O. Friedrichs and 
P. D. Lax in their classical work \cite{PeterLax} to study  boundary value problems of first order. The appearance of both works,  \cite{KohnNirenberg1965} and \cite{PeterLax}, in the same volume of Comm.  Pure  Appl. Math. shows an immediate profound impact of the  Kohn-Nirenberg calculus of pseudo-differential operators in the development of the solvability theory of partial differential operators.  Nevertheless, the pseudo-differential technique (which means, to solve problems in mathematics using microlocal methods), appeared before 1959 in the classical works of Mihlin, and in the theory of singular integrals  developed by Calder\'on and Zygmund approximating inverses of elliptic operators, and in 1959  in Calderon's proof  of Cauchy uniqueness for a wide  class of principal type operators, using a pseudo-differential factorization to prove a Carleman estimates \cite{Carleman}. 

Generalising  the Kohn-Nirenberg classes $S^m(\mathbb{R}^n\times \mathbb{R}^n),$ L. H\"ormander introduced in 1967, the classes $S^m_{\rho,\delta}(\mathbb{R}^n\times \mathbb{R}^n)$\footnote{which consists of all smooth functions $a$  satisfying $|\partial_{x}^\beta\partial_{\xi}^\alpha a(x,\xi)|=O(1+|\xi|)^{m-\rho|\alpha|+\delta|\beta|},$ when $|\xi|\rightarrow\infty.$}, $0\leq \delta<\rho\leq 1,$ motivated by construction of the parametrix $P$ of the heat operator $$\partial_t-\Delta_x$$ which has symbol $\sigma_P$ in the class $S^{-1}_{\frac{1}{2},\frac{1}{2}}(\mathbb{R}^n\times \mathbb{R}^n).$
At the same time,  in \cite{seeley} R. T. Seeley, developed the asymptotic expansions for the symbols of the complex powers of elliptic operators. Therefore, Seeley have tackled one of the fundamental problems of the functional calculus of pseudo-differential operators.  This can be regarded as  the first model for a functional calculus of pseudo-differential operators. The theory of pseudo-differential operators got a major breakthrough when they proved to be worthwhile in the proof of the seminal work of Atiyah and Singer on the index theory of elliptic operators on compact manifolds \cite{ASinger}. The ideas developed by Seeley for functional calculus \cite{seeley} have been carried forward by several researchers. We cite \cite{Shubin76, Shubin78, Beals, Buzanonikola,MeladzeShubin, CSS, KumanoTsutsumi73,Loya2003,Loya2003a,Sch86,Sch88,Ruzwirth} to mention a very few of them. In particular, in \cite{Shubin78}, the author has investigated only the case of complex powers of differential operators, in \cite{Beals} the structure of the inverse of an elliptic operator, which is also considered as the second fundamental problem of the functional calculus, has been examined by using different techniques  of Seeley.  In \cite{KumanoTsutsumi73} the formal theory of complex powers was developed by Kumano-go and Tsutsumi  using the similar methods as in \cite{seeley}, in \cite{MeladzeShubin} they developed functional calculus on connected unimodular Lie groups. On the other hand, the functional calculus for pseudo-differential operators on the manifolds with certain geometry (on the boundary or on the manifold itself) was studied in \cite{CSS,Sch86,Sch88,Loya2003,Loya2003a}. For example, Coriasco, Schrohe and Seiler \cite{CSS} looked over the bounded imaginary powers of differential operators on manifolds with conical singularities, Schrohe \cite{Sch88, Sch88} analysed the complex powers on noncompact manifolds and manifolds with fibered boundaries, Loya \cite{Loya2003} explored the manifolds with canonical singularity, where the author used the heat kernels techniques \cite{Loya2003a} in \cite{Loya2003}.

In 2014, the third author with J. Wirth \cite{Ruzwirth} developed the global functional calculus for the elliptic pseudo-differential operators on compact Lie groups using the globally defined matrix symbols instead of representations in local coordinates, which is the version of the analysis well adopted to the operator theory on compact Lie groups. The global theory of symbols and their calculus was introduced and investigated in detail by the third author and V. Turunen in \cite{RT, Ruzhansky-Turunen:IMRN}. H\"ormander classes on compact Lie groups were investigated by Ruzhansky, Turunen and Wirth \cite{Ruzhansky-Turunen-Wirth:JFAA} providing the characterisation of operators in H\"ormander's classes $S^m_{1,0}$ on the compact Lie group viewed as a manifold was given in terms of these matrix symbols, thus providing a link between local and global symbolic calculi. Matrix valued symbols also proved to be important in the study of the $L^p$-multipliers problems on general compact Lie groups \cite{RuzwirthLp}. The functional calculus (complex powers of elliptic operators) has several important applications to index theory, evolution equations, $\zeta$-functions of an operator, Wodzicki-type (non-commutative) residues, G\r{a}rding inequalities \cite{Ruz-Tur-Sharp}. We refer to \cite{Buzanonikola, Shubin76} for several aspects of the functional calculus and extensive reviews of the above topics and to \cite{CardonaRuzhansky2020} for the functional calculus of subelliptic pseudo-differential operators on compact Lie groups.

In this paper we work in the setting of the Fourier analysis arising from the spectral decomposition of a model operator $L_{\Omega}$ on a smooth manifold $\overline{\Omega}$ (with or without boundary) (\cite{RT2016,RT2017}). In order to address the problem of the functional calculus from the view of H\"ormander symbolic calculus developed by the last two authors \cite{RT2016} in this setting, we first introduce in Section \ref{Parameterellipticity} the concept of parameter dependent $L$-ellipticity with respect to an analytic curve in the complex plane in the setting of nonharmonic analysis and examine its properties. In general, for the pseudo-differential operators on manifolds one puts some restrictions on the H\"ormander symbol classes $S^m_{\rho, \delta}$  \cite[Section 4]{Shubin78}, usually one requires $1-\rho<\delta<\rho,$ which in turn gives $\rho>\frac{1}{2}.$ It is worth noting that, in this paper, we allow $0\leq \delta<\rho \leq 1.$ This freedom on the condition on $\rho$ and $\delta$ enables us to handle some specific classes of operator, for example the resolvent of an $L$-elliptic symbol, or its complex powers which cannot be handled by the standard theory due the restriction $\rho>\frac{1}{2}.$  

After developing the global functional calculus for pseudo-differential operators in the nonharmonic analysis setting on manifolds, we present applications of functional calculus to the G\r{a}rding inequality, $L^2$-boundedess of pseudo-differential operators and global solvability of evolution problems of hyperbolic/parabolic type on compact manifolds.  We will now discuss each application separately in detail.

The G\r{a}rding inequality on $\mathbb{R}^n$ is a well-known and important inequality for pseudo-differential operators on $\mathbb{R}^n$ with $0\leq \delta <\rho \leq 1$ (\cite{Taylor}) and on manifolds with the restriction $1-\rho<\delta<\rho.$ For Lie groups, this was proved for the operators in the H\"ormander class $(1,0)$-type using the results developed by Langlands \cite{Langland} for the semigroups of Lie groups. The third author and Wirth \cite{Ruzwirth} have obtained it for operators on compact Lie groups with matrix valued symbols under the condition $0\leq \delta<\rho \leq 1,$ using the global functional calculus developed by them. In this paper we carry forward the ideas of \cite{Ruzwirth} to establish the G\r{a}rding inequality for operators with the global H\"ormander symbols \cite{RT} under the same condition $0\leq \delta<\rho \leq 1,$ using the global functional calculus developed for pseudo-differential operator in Section \ref{SFC}. 

The second application of the functional calculus is to prove the $L^2$-boundedness of global operators associated to the  symbol class $S^0_{\rho, \delta}(\overline{\Omega} \times \mathcal{I})$ with $0 \leq \delta <\rho \leq 1$ which is presented in Section \ref{L^2boun}.  In Section \ref{GST}, we present an application of the G\r{a}rding inequality and consequently, an application of the functional calculus to study the existence and uniqueness of Cauchy problems.

\section{Preliminaries: global pseudo-differential calculus  associated with boundary value problems}\label{Sect2} 

Let $M=\overline{\Omega}$ be a  $C^\infty$-manifold with (possibly empty) boundary $\partial \Omega.$ Let us formulate some basics  of the non-harmonic analysis and the pseudo-differential calculus developed by the third and the fourth author in \cite{RT2016} (see also \cite{CardVishTokRuzI, DRT2017}): \begin{itemize}
    \item {\it{  Consider
a pseudo-differential operator $L:=L_{\Omega}$  of order $m$ 
on a smooth manifold ${\Omega}$ (in the sense of H\"ormander \cite{Hormander1985III}) equipped with some boundary
conditions \textnormal{(BC)} defining a  space of functions endowed with a complex structure of  vector space. We  will also assume the condition \textnormal{(BC+)}, which states that the boundary conditions define a closed topological space.\footnote{The assumption \textnormal{(BC)} may be reformulated by saying that the domain ${\rm Dom}(L)$ of the operator
$L$ is linear, and the condition (BC+) by saying that ${\rm Dom}(L)$ and ${\rm Dom}(L^*)$
are closed in the topologies of $C_{L}^\infty(\overline{\Omega})$ and  $C_{{L}^*}^\infty(\overline{\Omega})$, 
respectively, with the latter spaces and their topologies introduced in Definition \ref{TestFunSp}. }}}
\item {\it{  
The pseudo-differential operator ${L}_\Omega$ is assumed to have a  discrete spectrum
$\{\lambda_{\xi}\in\mathbb C: \, \xi\in\ind\}$
on $L^{2}(\Omega)$,
and we order the eigenvalues with
the occurring multiplicities in the ascending order:}}
$
 |\lambda_{j}|\leq|\lambda_{k}| \quad\textrm{ for } |j|\leq |k|\footnote{Let us denote by $u_{\xi}$  the eigenfunction of $L$ corresponding to the
eigenvalue $\lambda_{\xi}$ for each $\xi\in\ind$, so that
$ {L}u_{\xi}=\lambda_{\xi}u_{\xi},$ 
in $   \Omega,$ for all $  \xi\in\ind.
$
Here, the system of eigenfunctions $u_\xi$ satisfy the boundary conditions
\textnormal{(BC)} discussed earlier. The conjugate spectral problem is
$
L^{\ast}v_{\xi}=\overline{\lambda}_{\xi}v_{\xi},$ in $ \Omega$ for all $  \xi\in\ind,$
which we equip with the conjugate boundary conditions, which we may denote by \textnormal{(BC)}$^*$.
This adjoint problem is associated to the adjoint  $L^{*}:={L}_\Omega^*$ of $L.$ }.
$
\item { \it{ The eigenfunctions $u_\xi$ of ${L}$ (associated to $\lambda_\xi$) and $v_\xi$ of ${L}^*$  are considered to be $L^2$-normalised. Also, they  satisfy the condition of biorthogonality,
i.e.}}
$$
(u_{\xi},v_{\eta})_{L^2}=\delta_{\xi,\eta},\,\,$$ {\it{  where $\delta_{\xi,\eta}$  is the Kronecker-Delta and $(\cdot, \cdot)_{L^2}$ is the usual $L^2$-inner product given by }} $(f, g)_{L^{2}}:=\int\limits_{\Omega}f(x)\overline{g(x)}dx,$  $f,$  $g\in L^2(\Omega).$ 
\end{itemize}

From 
\cite{bari} it follows that the system $\{u_{\xi}: \,\,\,
\xi\in\ind\}$ is a basis in $L^{2}(\Omega)$ if and only if the
system $\{v_{\xi}: \,\,\, \xi\in\ind\}$ is a basis in
$L^{2}(\Omega)$.
So, from now, we assume the following:

\begin{itemize}
    \item \it{ 
The system
$\{u_{\xi}:  \xi\in\ind\}$ is a basis in $L^{2}(\Omega)$, i.e.
for every $f\in L^{2}(\Omega)$ there exists a unique series
$\sum_{\xi\in\ind} a_\xi u_\xi(x)$ that converges to $f$ in  $L^{2}(\Omega)$.
}
\end{itemize} 
Let us define the following notation ($L$-Japanese bracket)
\begin{equation}\label{EQ:angle}
\langle\xi\rangle:=(1+|\lambda_{\xi}|^2)^{\frac{1}{2m}},
\end{equation}
which will be used later in measuring the growth/decay of Fourier coefficients  of the distributions in our context.
Define the operator
${L}^{\circ}$ by setting its values on the basis $u_{\xi}$ by
\begin{equation}\label{EQ:Lo-def}
{L}^{\circ} u_{\xi}:=\overline{\lambda_{\xi}} u_{\xi},\quad
\textrm{ for all } \xi\in\ind,
\end{equation}
we can informally think of $\langle\xi\rangle,$ $\xi\in \mathcal{I},$
as of the eigenvalues of the positive (first order) pseudo-differential operator
$({\rm I}+{L^\circ\, L})^{\frac{1}{2m}}.$

\medskip
The following technical definition will be useful to single out the case when the eigenfunctions
of both $L$ and ${L}^*$ do not have zeros (WZ stands for `without zeros'):

\begin{definition}
\label{DEF: WZ-system}
The system $\{u_{\xi}: \,\,\, \xi\in\ind\}$ is called a ${\rm WZ}$-system if the functions $u_{\xi}(x), \,
v_{\xi}(x)$ do not have zeros on the domain $\overline{\Omega}$
for all $\xi\in\ind$, and if there exist $C>0$
and $N\geq0$ such that
$$
\inf\limits_{x\in\Omega}|u_{\xi}(x)|\geq
C\langle\xi\rangle^{-N},\,\,\,
\inf\limits_{x\in\Omega}|v_{\xi}(x)|\geq
C\langle\xi\rangle^{-N},
$$
as $\langle\xi\rangle\to\infty$. Here WZ stands for `without zeros'.
\end{definition}

In the sequel, unless stated otherwise, whenever we will use inverses $u_{\xi}^{-1}$
of the functions $u_{\xi}$,
we will suppose that the system $\{u_{\xi}: \, \xi\in\ind\}$ is a ${\rm WZ}$-system.
However, we will also try to mention explicitly when we make such an additional assumption.

\subsection{Global distributions generated by the boundary value problem}
\label{SEC:TD}

Now, we will present the spaces of distributions generated by the boundary value problem
${L}_\Omega$ and by its adjoint ${L}_\Omega^*$ and the related global Fourier analysis.
We first define the space $C_{{L}}^{\infty}(\overline{\Omega})$ of test functions.

\begin{definition}\label{TestFunSp}
The space
$C_{{L}}^{\infty}(\overline{\Omega}):={\rm Dom}({L}_\Omega^{\infty})$ is called the
space of test functions for ${L}_\Omega$. Here we define
$$
{\rm Dom}({L}_\Omega^{\infty}):=\bigcap_{k=1}^{\infty}{\rm Dom}({
L}_\Omega^{k}),
$$
where ${\rm Dom}({L}_\Omega^{k})$, or just ${\rm Dom}({L}^{k})$ for simplicity, is the domain of the
operator ${L}^{k}$, in turn defined as
$$
{\rm Dom}({L}^{k}):=\{f\in L^{2}(\Omega): \,\,\, {
L}^{j}f\in {\rm Dom}({L}), \,\,\, j=0, \,1, \, 2, \ldots,
k-1\}.
$$
The operators ${L}^{k}$, $k\in\mathbb N$, are  endowed with the same boundary conditions \textnormal{(BC)}.
The Fr\'echet topology of $C_{{L}}^{\infty}(\overline{\Omega})$ is given by the family of norms
\begin{equation}\label{EQ:L-top}
\|\varphi\|_{C^{k}_{{L}}}:=\max_{j\leq k}
\|{L}^{j}\varphi\|_{L^2(\Omega)}, \quad k\in\mathbb N_0,\,
 \varphi\in C_{{L}}^{\infty}(\overline{\Omega}).
\end{equation}

Analogously,  we introduce the space $C_{{
L^{\ast}}}^{\infty}(\overline{\Omega})$ corresponding to the adjoint operator ${L}_\Omega^*$ by
$$
C_{{L^{\ast}}}^{\infty}(\overline{\Omega}):=
{\rm Dom}(({L^{\ast}})^{\infty})=\bigcap_{k=1}^{\infty}{\rm
Dom}(({L^{\ast}})^{k}),
$$
where ${\rm Dom}(({L^{\ast}})^{k})$ is the domain of the
operator $({L^{\ast}})^{k}$,
$$
{\rm Dom}(({L^{\ast}})^{k}):=\{f\in L^{2}(\Omega): \,\,\, ({
L^{\ast}})^{j}f\in {\rm Dom}({L^{\ast}}), \,\,\, j=0, \ldots, k-1\},
$$
which satisfy the adjoint boundary conditions corresponding to the operator
${L}_\Omega^*$. The Fr\'echet topology of $C_{{L}^*}^{\infty}(\overline{\Omega})$ is given by the family of norms
\begin{equation}\label{EQ:L-top-adj}
\|\psi\|_{C^{k}_{{L}^*}}:=\max_{j\leq k}
\|({L}^*)^{j}\psi\|_{L^2(\Omega)}, \quad k\in\mathbb N_0,
\, \psi\in C_{{L}^*}^{\infty}(\overline{\Omega}).
\end{equation}
\end{definition}
\begin{remark}
If ${L}_\Omega$ is self-adjoint, i.e. if ${L}_\Omega^*={L}_\Omega$
with the equality of domains, then
$C_{{L^{\ast}}}^{\infty}(\overline{\Omega})=C_{{L}}^{\infty}(\overline{\Omega}).$ On the other hand, since we have $u_\xi\in C^\infty_{{L}}(\overline\Omega)$ and
$v_\xi\in C^\infty_{{L}^*}(\overline\Omega)$ for all $\xi\in\ind$, we observe that
the biorthogonality condition of the systems $\{u_\xi\}_{\xi\in \mathcal{I}},$ and $\{v_\xi\}_{\xi\in \mathcal{I}}$ implies that the spaces
$C^\infty_{{L}}(\overline\Omega)$ and $C^\infty_{{L}^*}(\overline\Omega)$
are dense in $L^2(\Omega)$.
\end{remark}

In general, for functions $f\in C_{{L}}^{\infty}(\overline{\Omega})$ and
$g\in C_{{L}^*}^{\infty}(\overline{\Omega})$, the $L^2$-duality makes sense in view
of the formula
\begin{equation}\label{EQ:duality}
({L}f, g)_{L^2(\Omega)}=(f,{L}^*g)_{L^2(\Omega)}.
\end{equation}
Therefore, in view of the formula \eqref{EQ:duality},
it makes sense to define the distributions $\mathcal D'_{{L}}(\Omega)$
as the space which is dual to $C_{{L}^*}^{\infty}(\overline{\Omega})$.
Note that the the respective boundary conditions of ${L}_\Omega$ and ${L}_\Omega^*$
are satisfied by the choice of $f$ and $g$ in corresponding domains.

\begin{definition}\label{DistrSp}
The space $$\mathcal D'_{{
L}}(\Omega):=\mathcal L(C_{{L}^*}^{\infty}(\overline{\Omega}),
\mathbb C)$$ of linear continuous functionals on
$C_{{L}^*}^{\infty}(\overline{\Omega})$ is called the space of
${L}$-distributions.\footnote{
We can understand the continuity here either in terms of the topology
\eqref{EQ:L-top-adj} or in terms of sequences, see
Proposition \ref{TH: UniBdd}.}
For
$w\in\mathcal D'_{{L}}(\Omega)$ and $\varphi\in C_{{L}^*}^{\infty}(\overline{\Omega})$,
we shall write
$$
w(\varphi)=\langle w, \varphi\rangle.
$$
Observe that, for any $\psi\in C_{{L}}^{\infty}(\overline{\Omega})$,
$$
C_{{L}^*}^{\infty}(\overline{\Omega})\ni \varphi\mapsto\int\limits_{\Omega}{\psi(x)} \, \varphi(x)\, dx
$$
is an ${L}$-distribution, which gives an embedding $\psi\in
C_{{L}}^{\infty}(\overline{\Omega})\hookrightarrow\mathcal D'_{{
L}}(\Omega)$.
We note that in the distributional notation formula \eqref{EQ:duality} becomes
\begin{equation}\label{EQ:duality-dist}
\langle{L}\psi, \varphi\rangle=\langle \psi,\overline{{L}^* \overline{\varphi}}\rangle.
\end{equation}
\end{definition}

With the topology on $C_{{L}}^{\infty}(\overline{\Omega})$
defined by \eqref{EQ:L-top},
the space $$\mathcal
D'_{{L^{\ast}}}(\Omega):=\mathcal L(C_{{L}}^{\infty}(\overline{\Omega}), \mathbb C)$$
of linear continuous functionals on $C_{{L}}^{\infty}(\overline{\Omega})$
is called the
space of ${L^{\ast}}$-distributions.

\begin{proposition}\label{TH: UniBdd}
A linear functional $w$ on
$C_{{L}^*}^{\infty}(\overline{\Omega})$ belongs to $\mathcal D'_{{
L}}(\Omega)$ if and only if there exists a constant $c>0$ and a
number $k\in\mathbb N_0$ with the property
\begin{equation}
\label{EQ: UnifBdd-s1} |w(\varphi)|\leq
c \|\varphi\|_{C^{k}_{{L}^*}} \quad \textrm{ for all } \, \varphi\in C_{{
L}^*}^{\infty}(\overline{\Omega}).
\end{equation}
\end{proposition}

The space $\mathcal D'_{{L}}(\Omega)$\footnote{The convergence in the linear space
$\mathcal D'_{{L}}(\Omega)$ is the usual weak convergence with respect to
the space $C_{{L}^*}^{\infty}(\overline{\Omega})$.} has many similarities with the
usual spaces of distributions. For example, suppose that for a linear continuous operator
$D:C_{{L}}^{\infty}(\overline{\Omega})\to C_{{L}}^{\infty}(\overline{\Omega})$
its adjoint $D^*$
preserves the adjoint boundary conditions (domain) of ${L}_\Omega^*$
and is continuous on the space
$C_{{L}^*}^{\infty}(\overline{\Omega})$, i.e.
that the operator
$D^*:C_{{L}^*}^{\infty}(\overline{\Omega})\to C_{{L}^*}^{\infty}(\overline{\Omega})$
is continuous.
Then we can extend $D$ to $\mathcal D'_{{L}}(\Omega)$ by
$$
\langle Dw,{\varphi}\rangle := \langle w, \overline{D^* \overline{\varphi}}\rangle \quad
(w\in \mathcal D'_{{L}}(\Omega),\,  \varphi\in C_{{
L}^*}^{\infty}(\overline{\Omega})).
$$
This extends \eqref{EQ:duality-dist} from $L$ to other operators.

The following principle of uniform boundedness is based on the
Banach--Steinhaus Theorem applied to the Fr\'echet space $C_{{
L}^*}^{\infty}(\overline{\Omega})$.

\begin{lemma} \label{LEM: UniformBoundedness}
Let $\{w_{j}\}_{j\in\mathbb N}$ be a sequence in $\mathcal
D'_{{L}}(\Omega)$ with the property that for every $\varphi\in
C_{{L}^*}^{\infty}(\overline{\Omega})$, the sequence
$\{w_{j}(\varphi)\}_{j\in\mathbb N}$ in $\mathbb C$ is bounded.
Then there exist constants $c>0$ and $k\in\mathbb N_0$ such that
\begin{equation}
\label{EQ: UniformBoundedness} |w_{j}(\varphi)|\leq c
\|\varphi\|_{C^{k}_{{L}^*}} \quad \textrm{ for all } \, j\in\mathbb N, \,\,
\varphi\in C_{{L}^*}^{\infty}(\overline{\Omega}).
\end{equation}
\end{lemma}

The lemma above leads to the following property of completeness of
the space of ${L}$-distributions. 

\begin{theorem} \label{TH: Com-nessDistr}
Let $\{w_{j}\}_{j\in\mathbb N}$ be a
sequence in $\mathcal D'_{{L}}(\Omega)$ with the property that
for every $\varphi\in C_{{L}^*}^{\infty}(\overline{\Omega})$ the
sequence $\{w_{j}(\varphi)\}_{j\in\mathbb N}$ converges in
$\mathbb C$ as $j\rightarrow\infty$. Denote the limit by
$w(\varphi)$.

{\rm (i)} Then $w:\varphi\mapsto w(\varphi)$ defines an ${
L}$-distribution on $\Omega$. Furthermore,
$$
\lim_{j\rightarrow\infty}w_{j}=w \,\,\,\,\,\,\,\, \hbox{in}
\,\,\,\,\,\,\, \mathcal D'_{{L}}(\Omega).
$$

{\rm (ii)} If $\varphi_{j}\rightarrow\varphi$ in $\in C_{{
L}^*}^{\infty}(\overline{\Omega})$, then
$$
\lim_{j\rightarrow\infty}w_{j}(\varphi_{j})=w(\varphi) \,\,\,\,\,\,\,\,
\hbox{in} \,\,\,\,\,\,\, \mathbb C.
$$
\end{theorem}

Similarly to the previous case, we have analogues of
Proposition \ref{TH: UniBdd} and Theorem \ref{TH: Com-nessDistr}
for ${L^{\ast}}$-distributions.

\subsection{${L}$-Fourier transform, ${L}$-Convolution, Plancherel formula, Sobolev spaces  and their Fourier images}
\label{SEC:FT}

Let us start by defining the ${L}$-Fourier transform introduced in \cite{RT2016}, which is generated by
the boundary value problem ${L}_\Omega$ and its main properties. Here we record that:

\MyQuote{assume that, with ${L}_0$ denoting ${L}$ or ${L}^*$, if
$f_j\in C_{{L}_0}^{\infty}(\overline{\Omega})$ satisfies $f_j\to f$ in $C_{{L}_0}^{\infty}(\overline{\Omega})$,
then $f\in C_{{L}_0}^{\infty}(\overline{\Omega})$.}{BC+}
\medskip
Let us denote by $\mathcal S(\ind)$  the space of rapidly decaying
functions $\varphi:\ind\rightarrow\mathbb C$.\footnote{That is,
$\varphi\in\mathcal S(\ind)$ if for any $M<\infty$ there
exists a constant $C_{\varphi, M}$ such that
$
|\varphi(\xi)|\leq C_{\varphi, M}\langle\xi\rangle^{-M}
$
holds for all $\xi\in\ind$.
The topology on $\mathcal
S(\ind)$ is given by the seminorms $p_{k}$, where
$k\in\mathbb N_{0}$ and $ p_{k}(\varphi):=\sup_{\xi\in\ind}\langle\xi\rangle^{k}|\varphi(\xi)|.$  }
In this space, the continuous linear functionals  are of
the form
$$
\varphi\mapsto\langle u, \varphi\rangle:=\sum_{\xi\in\ind}u(\xi)\varphi(\xi),
$$
where the functions $u:\ind \rightarrow \mathbb C$ grow at most
polynomially at infinity, i.e. there exist constants $M<\infty$
and $C_{u, M}$ such that
$ 
|u(\xi)|\leq C_{u, M}\langle\xi\rangle^{M},
$ 
holds for all $\xi\in\ind$. Such distributions $u:\ind
\rightarrow \mathbb C$ form the space of distributions which we denote by
$\mathcal S'(\ind)$.
We now define the $L$-Fourier transform on $C_{{L}}^{\infty}(\overline{\Omega})$.

\begin{definition} \label{FT}
We define the ${L}$-Fourier transform
$$
(\mathcal F_{{L}}f)(\xi)=(f\mapsto\widehat{f}):
C_{{L}}^{\infty}(\overline{\Omega})\rightarrow\mathcal S(\ind)
$$
by
\begin{equation}
\label{FourierTr}
\widehat{f}(\xi):=(\mathcal F_{{L}}f)(\xi)=\int\limits_{\Omega}f(x)\overline{v_{\xi}(x)}dx.
\end{equation}
Analogously, we define the ${L}^{\ast}$-Fourier
transform
$$
(\mathcal F_{{L}^{\ast}}f)(\xi)=(f\mapsto\widehat{f}_{\ast}):
C_{{L}^{\ast}}^{\infty}(\overline{\Omega})\rightarrow\mathcal
S(\ind)
$$
by
\begin{equation}\label{ConjFourierTr}
\widehat{f}_{\ast}(\xi):=(\mathcal F_{{
L}^{\ast}}f)(\xi)=\int\limits_{\Omega}f(x)\overline{u_{\xi}(x)}dx.
\end{equation}
\end{definition}

The expressions \eqref{FourierTr} and \eqref{ConjFourierTr}
are well-defined.
Moreover, we have:

\begin{proposition}\label{LEM: FTinS}
The ${L}$-Fourier transform
$\mathcal F_{{L}}$ is a bijective homeomorphism from $C_{{
L}}^{\infty}(\overline{\Omega})$ to $\mathcal S(\ind)$.
Its inverse  $$\mathcal F_{{L}}^{-1}: \mathcal S(\ind)
\rightarrow C_{{L}}^{\infty}(\overline{\Omega})$$ is given by
\begin{equation}
\label{InvFourierTr} (\mathcal F^{-1}_{{
L}}h)(x)=\sum_{\xi\in\ind}h(\xi)u_{\xi}(x),\quad h\in\mathcal S(\ind),
\end{equation}
so that the Fourier inversion formula becomes
\begin{equation}
\label{InvFourierTr0}
f(x)=\sum_{\xi\in\ind}\widehat{f}(\xi)u_{\xi}(x)
\quad \textrm{ for all } f\in C_{{
L}}^{\infty}(\overline{\Omega}).
\end{equation}
Similarly,  $\mathcal F_{{L}^{\ast}}:C_{{L}^{\ast}}^{\infty}(\overline{\Omega})\to \mathcal S(\ind)$
is a bijective homeomorphism and its inverse
$$\mathcal F_{{L}^{\ast}}^{-1}: \mathcal S(\ind)\rightarrow
C_{{L}^{\ast}}^{\infty}(\overline{\Omega})$$ is given by
\begin{equation}
\label{ConjInvFourierTr} (\mathcal F^{-1}_{{
L}^{\ast}}h)(x):=\sum_{\xi\in\ind}h(\xi)v_{\xi}(x), \quad h\in\mathcal S(\ind),
\end{equation}
so that the conjugate Fourier inversion formula becomes
\begin{equation}
\label{ConjInvFourierTr0} f(x)=\sum_{\xi\in\ind}\widehat{f}_{\ast}(\xi)v_{\xi}(x)\quad \textrm{ for all } f\in C_{{
L^*}}^{\infty}(\overline{\Omega}).
\end{equation}
\end{proposition}

By dualising the inverse ${L}$-Fourier
transform $\mathcal F_{{L}}^{-1}: \mathcal S(\ind)
\rightarrow C_{{L}}^{\infty}(\overline{\Omega})$, the
${L}$-Fourier transform extends uniquely to the mapping
$$\mathcal F_{{L}}: \mathcal D'_{{L}}(\Omega)\rightarrow
\mathcal S'(\ind)$$ by the formula
\begin{equation}\label{EQ: FTofDistr}
\langle\mathcal F_{{L}}w, \varphi\rangle:=
\langle w,\overline{\mathcal F_{{L}^*}^{-1}\overline{\varphi}}\rangle,
\quad\textrm{ with } w\in\mathcal D'_{{L}}(\Omega),\, \varphi\in\mathcal
S(\ind).
\end{equation}
It can be readily seen that if $w\in\mathcal D'_{{
L}}(\Omega)$ then $\widehat{w}\in\mathcal S'(\ind)$.
The reason for taking complex conjugates in \eqref{EQ: FTofDistr}
is that, if $w\in C_{{L}}^{\infty}(\overline{\Omega})$, we have
the equality
\begin{multline*}
\langle \widehat{w},\varphi\rangle =
\sum_{\xi\in\ind} \widehat{w}(\xi) \varphi(\xi)=
\sum_{\xi\in\ind} \left( \int_\Omega w(x) \overline{v_\xi(x)}dx\right) \varphi(\xi)\\
=
\int_\Omega w(x) \overline{\left( \sum_{\xi\in\ind} \overline{\varphi(\xi)} v_\xi(x)\right)} dx
=
\int_\Omega w(x) \overline{\left( \mathcal F_{{L}^*}^{-1} \overline{\varphi} \right)} dx
=\langle w,\overline{\mathcal F_{{L}^*}^{-1}\overline{\varphi}}\rangle.
\end{multline*}
Analogously, we have the mapping
$$\mathcal F_{{L}^*}: \mathcal D'_{{L}^*}(\Omega)\rightarrow
\mathcal S'(\ind)$$ defined by the formula
\begin{equation}\label{EQ: FTofDistr}
\langle\mathcal F_{{L}^*}w, \varphi\rangle:=
\langle w,\overline{\mathcal F_{{L}}^{-1}\overline{\varphi}}\rangle,
\quad\textrm{ with } w\in\mathcal D'_{{L}^*}(\Omega),\, \varphi\in\mathcal
S(\ind).
\end{equation}
It can be also seen that if $w\in\mathcal D'_{{
L}^*}(\Omega)$ then $\widehat{w}\in\mathcal S'(\ind)$.
The following statement follows from the work of Bari \cite[Theorem 9]{bari}:

\begin{lemma}\label{LEM: FTl2}
There exist constants $K,m,M>0$ such that for every $f\in L^{2}(\Omega)$
we have
$$
m^2\|f\|_{L^{2}}^2 \leq \sum_{\xi\in\ind} |\widehat{f}(\xi)|^2\leq M^2\|f\|_{L^{2}}^2 \leq \sum_{\xi\in\ind} |\widehat{f}_*(\xi)|^2\leq K^2\|f\|_{L^{2}}^2.
$$
\end{lemma}

However, we note that the Plancherel identity can be also achieved in suitably
defined $l^2$-spaces of Fourier coefficients, see Proposition \ref{PlanchId}.

Let us introduce a notion of the  ${L}$-convolution, an analogue of the convolution adapted to the
boundary problem ${L}_\Omega$.

\begin{definition} (${L}$-Convolution) \label{Convolution}
For $f, g\in C_{{L}}^{\infty}(\overline{\Omega})$ define their ${L}$-convolution
by
\begin{equation}\label{EQ: CONV1}
(f\sL g)(x):=\sum_{\xi\in\ind}\widehat{f}(\xi)\widehat{g}(\xi)u_{\xi}(x).
\end{equation}
By Proposition \ref{LEM: FTinS}  it is well-defined and we have
$f\sL g\in C_{{L}}^{\infty}(\overline{\Omega}).$\footnote{Due to the rapid decay of $L$-Fourier coefficients of
functions in $C_{{L}}^{\infty}(\overline{\Omega})$ compared to
a fixed polynomial growth of elements of $\mathcal S'(\ind)$, the
definition \eqref{EQ: CONV1}  still makes sense if $f\in \mathcal
D^\prime_{L}(\Omega)$ and $g\in C_{{
L}}^{\infty}(\overline{\Omega})$, with $f\sL g\in C_{{
L}}^{\infty}(\overline{\Omega}).$}
\end{definition}

Analogously to the ${L}$-convolution, we can introduce the ${L}^*$-convolution.
Thus, for $f, g\in C_{{
L^{\ast}}}^{\infty}(\overline{\Omega})$, we define the ${
L^{\ast}}$-convolution using the ${L}^*$-Fourier transform by
\begin{equation}\label{EQ: CONV2}
(f\sLs g)(x):=\sum_{\xi\in\ind}\widehat{f}_{\ast}(\xi)\widehat{g}_{\ast}(\xi)v_{\xi}(x).
\end{equation}
Its properties are similar to those of the ${L}$-convolution, so we may
formulate only the latter.
\begin{proposition}\label{ConvProp}For any $f, g\in C_{{
L}}^{\infty}(\overline{\Omega})$ we have
$$\widehat{f\sL g}=\widehat{f}\times \widehat{g},\,\xi\in \mathcal{I}.$$
The convolution is commutative and associative.
If $g \in C_{{
L}}^{\infty}(\overline{\Omega}),$ then for all
$f\in \mathcal D^\prime_{L}(\Omega)$ we have
\begin{equation}\label{EQ:conv1}
f\sL g\in C_{{L}}^{\infty}(\overline{\Omega}).
\end{equation}
In addition,  if $\Omega\subset \mathbb{R}^n$ is bounded, and $f,g\in  L^{2}(\Omega)$, then $f\sL g\in L^{1}(\Omega)$ with
$$\|f\sL g\|_{L^1}\leq C|\Omega|^{1/2} \|f\|_{L^2}\|g\|_{L^2},$$
where $|\Omega|\in (0,\infty]$ is the volume of $\Omega$, with $C$ independent
of $f,g,\Omega$.
\end{proposition}

\medskip
Let us denote by $ l^{2}_{{L}}=l^2({L})$  
the linear space of complex-valued functions $a$
on $\ind$ such that $\mathcal F^{-1}_{{L}}a\in
L^{2}(\Omega)$, i.e. if there exists $f\in L^{2}(\Omega)$ such that $\mathcal F_{{L}}f=a$.
Then the space of sequences $l^{2}_{{L}}$ is a
Hilbert space with the inner product
\begin{equation}\label{EQ: InnerProd SpSeq-s}
(a,\ b)_{l^{2}_{{
L}}}:=\sum_{\xi\in\ind}a(\xi)\ \overline{(\mathcal F_{{
L^{\ast}}}\circ\mathcal F^{-1}_{{L}}b)(\xi)}
\end{equation}
for arbitrary $a,\,b\in l^{2}_{{L}}$.
The norm of $l^{2}_{{L}}$ is then given by the
formula
\begin{equation}\label{EQ:l2norm}
\|a\|_{l^{2}_{{L}}}=\left(\sum_{\xi\in\ind}a(\xi)\
\overline{(\mathcal F_{{L^{\ast}}}\circ\mathcal F^{-1}_{{
L}}a)(\xi)}\right)^{1/2}, \quad \textrm{ for all } \, a\in l^{2}_{{L}}.
\end{equation} Analogously, we introduce the
Hilbert space $ l^{2}_{{L^{\ast}}}=l^{2}({L^{\ast}})$
as the space of functions $a$ on $\ind$
such that $\mathcal F^{-1}_{{L^{\ast}}}a\in L^{2}(\Omega)$,
with the inner product
\begin{equation}
\label{EQ: InnerProd SpSeq-s_2} (a,\ b)_{l^{2}_{{
L^{\ast}}}}:=\sum_{\xi\in\ind}a(\xi)\ \overline{(\mathcal
F_{{L}}\circ\mathcal F^{-1}_{{L^{\ast}}}b)(\xi)}
\end{equation}
for arbitrary $a,\,b\in l^{2}_{{L^{\ast}}}$. The norm of
$l^{2}_{{L^{\ast}}}$ is given by the formula
$$
\|a\|_{l^{2}_{{L^{\ast}}}}=\left(\sum_{\xi\in\ind}a(\xi)\
\overline{(\mathcal F_{{L}}\circ\mathcal F^{-1}_{{
L^{\ast}}}a)(\xi)}\right)^{1/2}
$$
for all $a\in l^{2}_{{L^{\ast}}}$. The spaces of sequences
$l^{2}_{{L}}$ and
$l^{2}_{{L^{\ast}}}$ are thus generated by biorthogonal systems
$\{u_{\xi}\}_{\xi\in\ind}$ and $\{v_{\xi}\}_{\xi\in\ind}$.
The reason for their definition in the above forms becomes clear again
in view of the following Plancherel identity:

\begin{proposition} {\rm(Plancherel's identity)}\label{PlanchId}
If $f,\,g\in L^{2}(\Omega)$ then
$\widehat{f},\,\widehat{g}\in l^{2}_{{L}}, \,\,\,
\widehat{f}_{\ast},\, \widehat{g}_{\ast}\in l^{2}_{{\rm
L^{\ast}}}$, and the inner products {\rm(\ref{EQ: InnerProd SpSeq-s}),
(\ref{EQ: InnerProd SpSeq-s_2})} take the form
$$
(\widehat{f},\ \widehat{g})_{l^{2}_{{L}}}=\sum_{\xi\in\ind}\widehat{f}(\xi)\ \overline{\widehat{g}_{\ast}(\xi)}
,\,\,
(\widehat{f}_{\ast},\ \widehat{g}_{\ast})_{l^{2}_{{
L^{\ast}}}}=\sum_{\xi\in\ind}\widehat{f}_{\ast}(\xi)\
\overline{\widehat{g}(\xi)}.
$$
In particular, we have
$$
\overline{(\widehat{f},\ \widehat{g})_{l^{2}_{{L}}}}=
(\widehat{g}_{\ast},\ \widehat{f}_{\ast})_{l^{2}_{{
L^{\ast}}}}.
$$
The Parseval identity takes the form
\begin{equation}\label{Parseval}
(f,g)_{L^{2}}=(\widehat{f},\widehat{g})_{l^{2}_{{
L}}}=\sum_{\xi\in\ind}\widehat{f}(\xi)\ \overline{\widehat{g}_{\ast}(\xi)}.
\end{equation}
Furthermore, for any $f\in L^{2}(\Omega)$, we have
$\widehat{f}\in l^{2}_{{L}}$, $\widehat{f}_{\ast}\in l^{2}_{{
L^{\ast}}}$, and
$ \|f\|_{L^{2}}=\|\widehat{f}\|_{l^{2}_{{
L}}}=\|\widehat{f}_{\ast}\|_{l^{2}_{{L^{\ast}}}}.
$
\end{proposition}

Now we introduce Sobolev spaces generated by the operator ${L}_{\Omega}$:

\begin{definition}[Sobolev spaces $\mathcal \mathcal{H}^{s}_{{L}}(\Omega)$] \label{SobSp}
For $f\in\mathcal D'_{{L}}(\Omega)\cap \mathcal D'_{{L}^{*}}(\Omega)$ and $s\in\mathbb R$, we say that
$$f\in\mathcal \mathcal{H}^{s}_{{L}}(\Omega)\, {\textrm{ if and only if }}\,
\langle\xi\rangle^{s}\widehat{f}(\xi)\in l^{2}_{{L}}.$$
We define the norm on $\mathcal \mathcal{H}^{s}_{{L}}(\Omega)$ by
\begin{equation}\label{SobNorm}
\|f\|_{\mathcal \mathcal{H}^{s}_{{
L}}(\Omega)}:=\left(\sum_{\xi\in\ind}
\langle\xi\rangle^{2s}\widehat{f}(\xi)\overline{\widehat{f}_{\ast}(\xi)}\right)^{1/2}.
\end{equation}
The Sobolev space $\mathcal \mathcal{H}^{s}_{{L}}(\Omega)$ is then the
space of ${L}$-distributions $f$ for which we have
$\|f\|_{\mathcal \mathcal{H}^{s}_{{L}}(\Omega)}<\infty$. Similarly,
we can define the
space $\mathcal \mathcal{H}^{s}_{{L^{\ast}}}(\Omega)$ by the
condition
\begin{equation}\label{SobNorm2}
\|f\|_{\mathcal \mathcal{H}^{s}_{{
L^{\ast}}}(\Omega)}:=\left(\sum_{\xi\in\ind}\langle\xi\rangle^{2s}\widehat{f}_{\ast}(\xi)\overline{\widehat{f}(\xi)}\right)^{1/2}<\infty.
\end{equation}
\end{definition}
We note that the expressions in \eqref{SobNorm} and 
\eqref{SobNorm2} are well-defined since the sum
$$
\sum_{\xi\in\ind}
\langle\xi\rangle^{2s}\widehat{f}(\xi)\overline{\widehat{f}_{\ast}(\xi)}=
(\langle\xi\rangle^{s}\widehat{f}(\xi),\langle\xi\rangle^{s}\widehat{f}(\xi))_{l^{2}_{L}}\geq 0
$$
is real and non-negative. 
Consequently, since we can write the sum in \eqref{SobNorm2} as the
complex conjugate of that in  \eqref{SobNorm}, and with both being real,
we see that the spaces $\mathcal \mathcal{H}^{s}_{{L}}(\Omega)$ and 
$\mathcal \mathcal{H}^{s}_{{L^{\ast}}}(\Omega)$ coincide as sets. Moreover, we have

\begin{proposition}\label{SobHilSpace}
For every $s\in\mathbb R$, the Sobolev space
$\mathcal \mathcal{H}^{s}_{{L}}(\Omega)$ is a Hilbert space with the
inner product
$$
(f,\ g)_{\mathcal \mathcal{H}^{s}_{{L}}(\Omega)}:=\sum_{\xi\in\ind
}\langle\xi\rangle^{2s}\widehat{f}(\xi)\overline{\widehat{g}_{\ast}(\xi)}.
$$
Similarly,
the Sobolev space
$\mathcal \mathcal{H}^{s}_{{L^{\ast}}}(\Omega)$ is a Hilbert space with
the inner product
$$
(f,\ g)_{\mathcal \mathcal{H}^{s}_{{
L^{\ast}}}(\Omega)}:=\sum_{\xi\in\ind}\langle\xi\rangle^{2s}\widehat{f}_{\ast}(\xi)\overline{\widehat{g}(\xi)}.
$$
For every $s\in\mathbb R$, the Sobolev spaces  $\mathcal \mathcal{H}^{s}_{{L}}(\Omega)$,
and $\mathcal \mathcal{H}^{s}_{{L}^*}(\Omega)$ are
isometrically isomorphic.
\end{proposition}

\subsection{$L$-Schwartz kernel theorem}
\label{SEC:Schwartz}

This subsection is devoted to discuss the Schwartz kernel theorem in the space
of distributions $\mathcal D'_{{L}}(\Omega).$ 
In this analysis we will need the following assumption which may be
also regarded as the definition of the number $s_{0}$. So, from now on we will make the
following:

\begin{assump}
\label{Assumption_4}
Assume that the number
$s_0\in\mathbb R$ is such that we have
$$\sum_{\xi\in\ind} \langle\xi\rangle^{-s_0}<\infty.$$
\end{assump}
Recalling the operator ${L}^{\circ}$ in \eqref{EQ:Lo-def}
the assumption \eqref{Assumption_4} is equivalent to assuming that
the operator $({\rm I}+{L^\circ L})^{-\frac{s_0}{4m}}$ is Hilbert-Schmidt on $L^2(\Omega)$.\\

{\it{
Indeed, recalling the definition of $\langle\xi\rangle$ in \eqref{EQ:angle},
namely that $\langle\xi\rangle$ are the eigenvalues of $({\rm I}+{L^\circ L})^{-\frac{s_0}{2m}}$,
that the operator $({\rm I}+{L^\circ L})^{-\frac{s_0}{4m}}$ is Hilbert-Schmidt on $L^2(\Omega),$ is equivalent to
the condition that}}
\begin{equation}\label{EQ:HS-conv}
\|({\rm I}+{L^\circ L})^{-\frac{s_0}{4m}}\|_{\tt HS}^2\cong \sum_{\xi\in\ind}
\langle\xi\rangle^{-s_0}<\infty.
\end{equation}
\begin{remark}If $L$ is elliptic, we may expect that we can take any $s_0>n:=\textnormal{dim}(\Omega)$ but this depends on the
boundary conditions in general.
The order $s_0$ will enter the regularity properties of the Schwartz kernels.
\end{remark}
We will use the notation
$$C^{\infty}_{{L}}(\overline{\Omega}\times \overline{\Omega}):=
C^{\infty}_{{L}}(\overline{\Omega})\otimes C^{\infty}_{{L}}(\overline{\Omega}),$$
and for the corresponding dual space we write
$ \mathcal D'_{{L}}(\Omega\times\Omega):=
\left(C^{\infty}_{{L}}(\overline{\Omega}\times \overline{\Omega})\right)^\prime.$
By following \cite{RT2016},
to any continuous linear operator  $A:C^{\infty}_{{
L}}(\overline{\Omega})\rightarrow\mathcal D'_{{L}}(\Omega),$ we can associate a  kernel $K\in \mathcal D'_{{L}}(\Omega\times\Omega)$ such that
$$
\langle Af,g\rangle=\int\limits_{\Omega}\int\limits_{\Omega}K(x,y)f(x)g(y)dxdy,
$$
and, using the notion of the ${L}$-convolution, also a convolution kernel $k_{A}(x)\in\mathcal D'_{{L}}(\Omega)$, such that
$$
Af(x)=(k_{A}(x)\sL f)(x),
$$ provided that  $\{u_{\xi}: \,\,\, \xi\in\ind\}$ is a ${\rm WZ}$-system
in the sense of Definition \ref{DEF: WZ-system}.
As usual, $K_{A}$ is called the Schwartz kernel of $A$. Note that, by using the Fourier series formula for $f\in C^{\infty}_{{L}}(\overline{\Omega})$, 
$$
f(y)=\sum\limits_{\eta\in\ind}\widehat{f}(\eta) u_{\eta}(y),
$$
we can also write
\begin{equation}\label{EQ:int2}
Af(x)=\sum\limits_{\eta\in\ind}\widehat{f}(\eta)\int\limits_{\Omega}K_{A}(x,y)u_{\eta}(y)dy,
\end{equation}  
and the ${L}$-distribution $k_{A}\in\mathcal D'_{{L}}(\Omega\times\Omega)$ is determined
by the formula
\begin{equation} \label{EQ: KernelConv}
k_{A}(x,z):=k_A(x)(z):=\sum\limits_{\eta\in\ind}u_{\eta}^{-1}(x)
\int\limits_{\Omega}K_{A}(x,y)u_{\eta}(y)dy \,
u_{\eta}(z).
\end{equation}
Since for some $C>0$ and $N\geq0,$ we have, by Definition \ref{DEF: WZ-system},
$$
\inf\limits_{x\in\overline{\Omega}}|u_{\eta}(x)|\geq
C\langle\eta\rangle^{-N},
$$
the series in (\ref{EQ: KernelConv}) is converges in the sense of
${L}$-distributions.
Formula \eqref{EQ: KernelConv} means that the Fourier transform of $k_A$ in the
second variable satisfies
\begin{equation}\label{EQ:int3}
 \widehat{k}_{A} (x,\eta)u_{\eta}(x)=
\int\limits_{\Omega}K_{A}(x,y)u_{\eta}(y)dy.
\end{equation}
Combining this and \eqref{EQ:int2} we get
\begin{equation*}
Af(x)=\sum\limits_{\eta\in\ind}
\widehat{f}(\eta)\int\limits_{\Omega}K_{A}(x,y)u_{\eta}(y)dy
=\sum\limits_{\eta\in\ind}\widehat{f}(\eta)\widehat{k}_{A}(x,\eta)u_{\eta}(x)
=(f\sL k_{A}(x))(x),
\end{equation*}
where in the last equality we used the notion of the $L$-convolution in Definition \ref{Convolution}.

\subsection{${L}$-Quantization and and full symbols}
\label{SEC:quantization}

In this subsection we describe the ${L}$-quantization induced by the boundary value problem
${L}_\Omega$. From now on we will assume
that the system of functions $\{u_{\xi}:\, \xi\in\ind\}$ is a ${\rm WZ}$-system
in the sense of Definition \ref{DEF: WZ-system}. Later, we will make some remarks
on what happens when this assumption is not satisfied.

\begin{definition}[${L}$-Symbols of operators on $\Omega$] \label{$L$--Symbols}
The ${L}$-symbol of a linear continuous
operator $$A:C^{\infty}_{{L}}(\overline{\Omega})\rightarrow
\mathcal D'_{{L}}(\Omega)$$ at $x\in\Omega$ and
$\xi\in\ind$ is defined by
$$\sigma_{A}(x, \xi):=\widehat{k_{A}(x)}(\xi)=\mathcal F_{{L}}(k_{A}(x))(\xi).$$
Hence, we can also write
$$\sigma_{A}(x, \xi)=\int\limits_{\Omega}k_{A}(x,y)\overline{v_{\xi}(y)}dy=
\langle k_{A}(x),\overline{v_{\xi}}\rangle.$$
\end{definition}

By the ${L}$-Fourier inversion formula the convolution kernel
can be recovered from the symbol:
\begin{equation}
\label{Kernel} k_{A}(x, y)=\sum_{\xi\in\ind}\sigma_{A}(x,
\xi)u_{\xi}(y),
\end{equation}
all in the sense of ${L}$-distributions. We now show that an
operator $A$ can be represented by its symbol \cite{RT2016}.

\begin{theorem}[${L}$--quantization] \label{QuanOper}
Let $ A:C^{\infty}_{{L}}(\overline{\Omega})\rightarrow
C^{\infty}_{{L}}(\overline{\Omega})$ be a continuous linear
operator with {L}-symbol $\sigma_{A}$. Then
\begin{equation}\label{Quantization}
Af(x)=\sum_{\xi\in\ind}
u_{\xi}(x)\sigma_{A}(x, \xi) \widehat{f}(\xi)
\end{equation}
for every $f\in C^{\infty}_{{L}}(\overline{\Omega})$ and
$x\in\Omega$.
The {L}-symbol $\sigma_{A}$ satisfies
\begin{equation}\label{FormSymb}
\sigma_{A}(x,\xi)=u_{\xi}(x)^{-1}(Au_{\xi})(x)
\end{equation}
for all $x\in\Omega$ and $\xi\in\ind$.
\end{theorem}

Now,
we  collect several formulae for the symbol under the assumption that
the biorthogonal system $u_\xi$ is a WZ-system:\footnote{In the case when $\{u_{\xi}: \, \xi\in\ind\}$ is not a
WZ-system, we can still understand the ${L}$-symbol
$\sigma_{A}$ of the operator $A$ as a function on
$\overline{\Omega}\times\ind$, for which the equality
$
u_{\xi}(x)\sigma_{A}(x,\xi)=\int\limits_{\Omega}K_{A}(x,y)u_{\xi}(y)dy
$
holds for all $\xi$ in $\ind$ and for $x\in\overline{\Omega}$.
Of course, this implies certain restrictions on the zeros of the
Schwartz kernel $K_A$. Such restrictions may be considered
natural from the point of view of the scope of problems that
can be treated by our approach in the case when the eigenfunctions
$u_\xi(x)$ may vanish at some points $x$. We refer to \cite{RT2017} for the calculus without the WZ-condition.}

\begin{corollary}\label{COR: SymFor}
We have the following equivalent formulae for {L}-symbols:
\begin{align*}
{\rm (i)} \,\,\,\,\, \sigma_{A}(x, \xi)
&=\int\limits_{\Omega}k_{A}(x,y)\overline{v_{\xi}(y)}dy;\\
{\rm (ii)} \,\,\,\,\, \sigma_{A}(x, \xi)&=u_{\xi}^{-1}(x)(Au_{\xi})(x);\\
{\rm (iii)} \,\,\,\,\,
\sigma_{A}(x,\xi)&=u_{\xi}^{-1}(x)\int\limits_{\Omega}K_{A}(x,y)u_{\xi}(y)dy;\\
{\rm (iv)} \,\,\,\,\, \sigma_{A}(x,
\xi)&=u_{\xi}^{-1}(x)\int\limits_{\Omega}\int\limits_{\Omega}F(x,y,z)k_{A}(x,y)
u_{\xi}(z)dydz.
\end{align*}
Here and in the sequel we write $u_{\xi}^{-1}(x)=u_{\xi}(x)^{-1}.$
Formula {\rm (iii)} also implies
\begin{align*}
{\rm (v)} \,\,\,\,\, K_A(x,y)&=
\sum_{\xi\in\ind}  u_\xi(x) \sigma_A(x,\xi) \overline{v_\xi(y)}.
\end{align*}
\end{corollary}

Similarly, we can introduce an analogous notion of the ${L^{\ast}}$-quantization.

\begin{definition}[${L^{\ast}}$-Symbols of operators on $\Omega$] \label{$L$--Symbols_star}
The ${L^{\ast}}$-symbol of a linear
continuous operator $$A:C^{\infty}_{{
L^{\ast}}}(\overline{\Omega})\rightarrow \mathcal D'_{{L}^*}(\Omega)$$ 
at $x\in{\Omega}$ and
$\xi\in\ind$ is defined by
$$\tau_{A}(x, \xi):=\mathcal F_{{L^{\ast}}}(\widetilde{k}_{A}(x))(\xi).$$
We can also write
$$\tau_{A}(x, \xi)=\int\limits_{\Omega}\widetilde{k}_{A}(x,y)\overline{u_{\xi}(y)}dy=
\langle\widetilde{k}_{A}(x),\overline{u_{\xi}}\rangle.$$
\end{definition}

By the ${L^{\ast}}$-Fourier inversion formula the convolution
kernel can be regained from the symbol:
\begin{equation}
\label{Kernel_star} \widetilde{k}_{A}(x, y)=\sum_{\xi\in\ind}\tau_{A}(x, \xi)v_{\xi}(y)
\end{equation}
in the sense of ${L^{\ast}}$--distributions.
Analogously to the $L$-quantization, we have:

\begin{corollary}[${L^{\ast}}$-quantization] \label{QuanOper_star}
Let $\tau_{A}$ be the ${L}^*$-symbol of a
continuous linear operator $ A:C^{\infty}_{{
L^{\ast}}}(\overline{\Omega})\rightarrow C^{\infty}_{{
L^{\ast}}}(\overline{\Omega}).$  Then
\begin{equation}
\label{Quantization_star} Af(x)=\sum_{\xi\in\ind}
v_{\xi}(x) \tau_{A}(x, \xi)  \widehat{f}_{\ast}(\xi)
\end{equation}
for every $f\in C^{\infty}_{{L^{\ast}}}(\overline{\Omega})$
and $x\in\Omega$.
For all
$x\in\Omega$ and $\xi\in\ind$, we have
\begin{equation}
\label{FormSymb_star} \tau_{A}(x,
\xi)=v_{\xi}(x)^{-1}(Av_{\xi})(x).
\end{equation}
We also have the following equivalent formulae for the ${L^{\ast}}$-symbol:
\begin{align*}
{\rm (i)} \,\,\,\,\, \tau_{A}(x, \xi)
&=\int\limits_{\Omega}\widetilde{k}_{A}(x,y)\overline{u_{\xi}(y)}dy;\\
{\rm (ii)} \,\,\,\,\,
\tau_{A}(x,\xi)&=v_{\xi}^{-1}(x)\int\limits_{\Omega}\widetilde{K}_{A}(x,y)v_{\xi}(y)dy.\\
\end{align*}
\end{corollary}

\subsection{Difference operators and symbolic calculus}
\label{SEC:differences}
In this subsection we discuss difference operators that will be instrumental in defining symbol
classes for the symbolic calculus of operators. 

Let $q_{j}\in C^{\infty}({\Omega}\times{\Omega})$, $j=1,\ldots,l$, be a given family
of smooth functions.
We will call the collection of $q_j$'s { $L$-strongly admissible} if the following properties hold:
\begin{itemize}
\item For every $x\in\Omega$, the multiplication by $q_{j}(x,\cdot)$
is a continuous linear mapping on
 $C^{\infty}_{{L}}(\overline{\Omega})$, for all $j=1,\ldots,l$;
\item $q_{j}(x,x)=0$ for all $j=1,\ldots,l$;
\item $
{\rm rank}(\nabla_{y}q_{1}(x,y), \ldots, \nabla_{y}q_{l}(x,y))|_{y=x}=n;
$
\item  the diagonal in $\Omega\times\Omega$ is the only set when all of
$q_j$'s vanish:
$$
\bigcap_{j=1}^l \left\{(x,y)\in\Omega\times\Omega: \, q_j(x,y)=0\right\}=\{(x,x):\, x\in\Omega\}.
$$
\end{itemize}

We note that the first property above implies that for every $x\in\Omega$, the multiplication by
$q_{j}(x,\cdot)$ is also well-defined and extends to a continuous linear mapping on
$\mathcal D'_{{L}}(\Omega)$. Also, the last property above contains the second one
but we chose to still give it explicitly for the clarity of the exposition.

The collection of $q_j$'s with the above properties generalises the notion of a strongly
admissible collection of functions for difference operators introduced in
\cite{Ruzhansky-Turunen-Wirth:JFAA} in the context of compact Lie groups.
We will use the multi-index notation
$$
q^{\alpha}(x,y):=q^{\alpha_1}_{1}(x,y)\cdots q^{\alpha_l}_{l}(x,y).
$$
Analogously, the notion of an ${L}^{*}$-strongly admissible collection suitable for the
conjugate problem is that of a family
$\widetilde{q}_{j}\in C^{\infty}({\Omega}\times{\Omega})$, $j=1,\ldots,l$, satisfying the properties:
\begin{itemize}
\item For every $x\in\Omega$, the multiplication by $\widetilde{q}_{j}(x,\cdot)$
is a continuous linear mapping on
 $C^{\infty}_{{L}^{*}}(\overline{\Omega})$, for all $j=1,\ldots,l$;
\item $\widetilde{q}_{j}(x,x)=0$ for all $j=1,\ldots,l$;
\item $
{\rm rank}(\nabla_{y}\widetilde{q}_{1}(x,y), \ldots, \nabla_{y}\widetilde{q}_{l}(x,y))|_{y=x}=n;
$
\item  the diagonal in $\Omega\times\Omega$ is the only set when all of
$\widetilde{q}_j$'s vanish:
$$
\bigcap_{j=1}^l \left\{(x,y)\in\Omega\times\Omega: \, \widetilde{q}_j(x,y)=0\right\}=\{(x,x):\, x\in\Omega\}.
$$
\end{itemize}
We also write
$$
\widetilde{q}^{\alpha}(x,y):=\widetilde{q}^{\alpha_1}_{1}(x,y)\cdots
\widetilde{q}^{\alpha_l}_{l}(x,y).
$$
We now record the Taylor expansion formula with respect to a family of $q_j$'s,
which follows from expansions of functions $g$ and
$q^{\alpha}(e,\cdot)$ by the common Taylor series   \cite{RT2016}:

\begin{proposition}\label{TaylorExp}
Any smooth function $g\in C^{\infty}({\Omega})$ can be
approximated by Taylor polynomial type expansions, i.e. for $e\in\Omega$, we have
$$g(x)=\sum_{|\alpha|<
N}\frac{1}{\alpha!}D^{(\alpha)}_{x}g(x)|_{x=e}\, q^{\alpha}(e,x)+\sum_{|\alpha|=
N}\frac{1}{\alpha!}q^{\alpha}(e,x)g_{N}(x)
$$
\begin{equation}
\sim\sum_{\alpha\geq
0}\frac{1}{\alpha!}D^{(\alpha)}_{x}g(x)|_{x=e}\, q^{\alpha}(e,x)
\label{TaylorExpFormula}
\end{equation}
in a neighborhood of $e\in\Omega$, where $g_{N}\in
C^{\infty}({\Omega})$ and
$D^{(\alpha)}_{x}g(x)|_{x=e}$ can be found from the recurrent formulae:
$D^{(0,\cdots,0)}_{x}:=I$ and for $\alpha\in\mathbb N_0^l$,
$$
\mathsf
\partial^{\beta}_{x}g(x)|_{x=e}=\sum_{|\alpha|\leq|\beta|}\frac{1}{\alpha!}
\left[\mathsf
\partial^{\beta}_{x}q^{\alpha}(e,x)\right]\Big|_{x=e}D^{(\alpha)}_{x}g(x)|_{x=e},
$$
where $\beta=(\beta_1, \ldots, \beta_n)$ and
$
\partial^{\beta}_{x}=\frac{\partial^{\beta_{1}}}{\partial x_{1}^{\beta_{1}}}\cdots
\frac{\partial^{\beta_{n}}}{\partial x_{n}^{\beta_{n}}}.
$
\end{proposition}

Analogously, any function $g\in C^{\infty}({\Omega})$
can be approximated by Taylor polynomial type expansions corresponding to
the adjoint problem, i.e. we
have
$$g(x)=\sum_{|\alpha|<
N}\frac{1}{\alpha!}\widetilde{D}^{(\alpha)}_{x}g(x)|_{x=e}\, \widetilde{q}^{\alpha}(e,x)+\sum_{|\alpha|=
N}\frac{1}{\alpha!}\widetilde{q}^{\alpha}(e,x)g_{N}(x)
$$
\begin{equation}
\sim\sum_{\alpha\geq
0}\frac{1}{\alpha!}\widetilde{D}^{(\alpha)}_{x}g(x)|_{x=e}\, \widetilde{q}^{\alpha}(e,x)
\label{TaylorExpFormula}
\end{equation}
in a neighborhood of $e\in\Omega$, where $g_{N}\in
C^{\infty}({\Omega})$ and
$\widetilde{D}^{(\alpha)}_{x}g(x)|_{x=e}$ are found from the
recurrent formula: $\widetilde{D}^{(0,\cdots,0)}:=I$ and for
$\alpha\in\mathbb N_{0}^{l}$,
$$
\partial^{\beta}_{x}g(x)|_{x=e}=\sum_{|\alpha|\leq|\beta|}\frac{1}{\alpha!}
\left[
\partial^{k}_{x}\widetilde{q}^{\alpha}(e,x)\right]\Big|_{x=e}\widetilde{D}^{(\alpha)}_{x}g(x)|_{x=e},
$$
where $\beta=(\beta_1, \ldots, \beta_n)$, and $\partial^{\beta}$ is defined as in
Proposition \ref{TaylorExp}.

It can be seen that operators $D^{(\alpha)}$ and
$\widetilde{D}^{(\alpha)}$ are differential operators of order
$|\alpha|$.
We now define difference operators acting on Fourier coefficients.
Since the problem in general may lack any invariance or symmetry structure,
the introduced difference operators will depend on a point $x$ where they
will be taken when applied to symbols.

\begin{definition}\label{DEF: DifferenceOper}\label{DEF: DifferenceOper_2}
For WZ-systems, we define difference operator $\Delta_{q,(x)}^{\alpha}$ acting
on Fourier coefficients by any of the following equal expressions
\begin{align*}
\Delta_{q,(x)}^{\alpha}\widehat{f}(\xi)& = u_{\xi}^{-1}(x)
\int\limits_{\Omega}\Big[\int\limits_{\Omega}q^{\alpha}(x,y)F(x,y,z)f(z)dz\Big]u_{\xi}(y)dy
\\
& = u_{\xi}^{-1}(x)
\sum_{\eta\in\ind}\mathcal
F_{L}\Big(q^{\alpha}(x,\cdot)u_{\xi}(\cdot)\Big)(\eta)\widehat{f}(\eta)u_{\eta}(x)
\\
& = u_{\xi}^{-1}(x) \left([q^{\alpha}(x,\cdot)u_{\xi}(\cdot)]\sL
f\right)(x).
\end{align*}
Analogously, we define the difference operator
$\widetilde{\Delta}_{q,(x)}^{\alpha}$ acting on adjoint Fourier
coefficients by
$$
\widetilde{\Delta}_{\widetilde{q},(x)}^{\alpha}\widehat{f}_{\ast}(\xi):=
v_{\xi}^{-1}(x)\sum_{\eta\in\ind}\mathcal F_{{
L^{\ast}}}\Big(\widetilde{q}^{\alpha}(x,\cdot)v_{\xi}(\cdot)\Big)(\eta)\widehat{f}_{\ast}(\eta)v_{\eta}(x).
$$
\end{definition}
For simplicity, if there is no confusion, for a fixed collection
of $q_j$'s, instead of $\Delta_{q,(x)}$ and
$\widetilde{\Delta}_{\widetilde{q},(x)}$ we will often simply write
$\Delta_{(x)}$ and $\widetilde{\Delta}_{(x)}$.

\begin{remark}
Applying difference operators to a symbol
and using formulae from Section \ref{SEC:quantization}, we
obtain
\begin{align}\nonumber
\Delta_{(x)}^{\alpha}a(x,\xi)&=u_{\xi}^{-1}(x)\sum_{\eta\in\ind}\mathcal
F_{L}\Big(q^{\alpha}(x,\cdot)u_{\xi}(\cdot)\Big)(\eta)a(x,\eta)u_{\eta}(x)\\ \nonumber
&=u_{\xi}^{-1}(x)\sum_{\eta\in\ind}\mathcal
F_{L}\Big(q^{\alpha}(x,\cdot)u_{\xi}(\cdot)\Big)(\eta)\int\limits_{\Omega}K(x,y)u_{\eta}(y)dy\\ \nonumber
&=u_{\xi}^{-1}(x)\int\limits_{\Omega}K(x,y)\left[\sum_{\eta\in\ind}\mathcal
F_{L}\Big(q^{\alpha}(x,\cdot)u_{\xi}(\cdot)\Big)(\eta)u_{\eta}(y)\right]dy\\
&=u_{\xi}^{-1}(x)\int\limits_{\Omega}q^{\alpha}(x,y)K(x,y)u_{\xi}(y)dy. \label{EQ:diff-symb}
\end{align}
In view of the first property of the strongly admissible collections, for each
$x\in\Omega$, the multiplication
by $q^{\alpha}(x,\cdot)$ is well defined on $\mathcal D'_{{L}}(\Omega)$.
Therefore, we can write \eqref{EQ:diff-symb} also in the distributional form
$$
 \Delta_{(x)}^{\alpha}a(x,\xi) = u_{\xi}^{-1}(x)\,
 \langle q^{\alpha}(x,\cdot)K(x,\cdot),u_{\xi}\rangle.
$$
\end{remark}

Plugging the expression (v) from Corollary \ref{COR: SymFor} for the kernel in terms of the symbol
into \eqref{EQ:diff-symb}, namely, using
$$
K(x,y)=
\sum_{\eta\in\ind}  u_\eta(x) a(x,\eta) \overline{v_\eta(y)},
$$
we record another useful form of \eqref{EQ:diff-symb} to be used later as
\begin{align}\nonumber
 \Delta_{(x)}^{\alpha}a(x,\xi) &=
 u_{\xi}^{-1}(x)\int\limits_{\Omega}q^{\alpha}(x,y)\left[
 \sum_{\eta\in\ind}  u_\eta(x) a(x,\eta) \overline{v_\eta(y)} \right]u_{\xi}(y)dy \\
 & = u_{\xi}^{-1}(x)   \sum_{\eta\in\ind}u_\eta(x) a(x,\eta)
 \left[ \int\limits_{\Omega}q^{\alpha}(x,y) \overline{v_\eta(y)} u_{\xi}(y)dy\right],
 \label{EQ:diff-symb-2}
\end{align}
with the usual distributional interpretation of all the steps. In the sequel we will also require
the ${L^*}$-version of this formula, which we record now as
\begin{equation}\label{EQ:diff-symb-3}
  \widetilde\Delta_{(x)}^{\alpha}a(x,\xi)=
  v_{\xi}^{-1}(x)   \sum_{\eta\in\ind}v_\eta(x) a(x,\eta)
 \left[ \int\limits_{\Omega}\widetilde q^{\alpha}(x,y) \overline{u_\eta(y)} v_{\xi}(y)dy\right].
\end{equation}

Using such difference operators and derivatives $D^{(\alpha)}$ from
Proposition \ref{TaylorExp}
we can now define classes of symbols.

\begin{definition}[Symbol class $S^m_{\rho,\delta}(\overline{\Omega}\times\ind)$]\label{DEF: SymClass}
Let $m\in\mathbb R$ and $0\leq\delta,\rho\leq 1$. The ${L}$-symbol class
$S^m_{\rho,\delta}(\overline{\Omega}\times\ind)$ consists of
those functions $a(x,\xi)$ which are smooth in $x$ for all
$\xi\in\ind$, and which satisfy
\begin{equation}\label{EQ:symbol-class}
  \left|\Delta_{(x)}^\alpha D^{(\beta)}_{x} a(x,\xi) \right|
        \leq C_{a\alpha\beta m}
                \ \langle\xi\rangle^{m-\rho|\alpha|+\delta|\beta|}
\end{equation}
for all $x\in\overline{\Omega}$, for all $\alpha,\beta\geq 0$, and for all $\xi\in\ind$.
Here the operators $D^{(\beta)}_{x}$ are defined in Proposition
\ref{TaylorExp}. We will often denote them simply by $D^{(\beta)}$.

The class $S^m_{1,0}(\overline{\Omega}\times\ind)$ will be often
denoted by writing simply $S^m(\overline{\Omega}\times\ind)$.
In \eqref{EQ:symbol-class}, we assume that the inequality is satisfied for $x\in\Omega$ and
it extends to the closure $\overline\Omega$.
Furthermore, we define
$$
S^{\infty}_{\rho,\delta}(\overline{\Omega}\times\ind):=\bigcup\limits_{m\in\mathbb
R}S^{m}_{\rho,\delta}(\overline{\Omega}\times\ind)
$$
and
$$
S^{-\infty}(\overline{\Omega}\times\ind):=\bigcap\limits_{m\in\mathbb
R}S^{m}(\overline{\Omega}\times\ind).
$$
When we have two $L$-strongly admissible collections, expressing one in terms of
the other similarly to Proposition \ref{TaylorExp} and arguing similarly to
\cite{Ruzhansky-Turunen-Wirth:JFAA}, we can convince ourselves that for $\rho>\delta$ the
definition of the symbol class does not depend on the choice of an
 $L$-strongly admissible collection.

Analogously, we define the ${L^{\ast}}$-symbol class
$\widetilde{S}^m_{\rho,\delta}(\overline{\Omega}\times\ind)$
as the space of those functions $a(x,\xi)$ which are smooth in $x$ for
all $\xi\in\ind$, and which satisfy
\begin{equation*}
  \left|\widetilde{\Delta}_{(x)}^\alpha \widetilde{D}^{(\beta)} a(x,\xi) \right|
        \leq C_{a\alpha\beta m}
                \ \langle\xi\rangle^{m-\rho|\alpha|+\delta|\beta|}
\end{equation*}
for all $x\in\overline{\Omega}$, for all $\alpha,\beta\geq 0$, and for all $\xi\in\ind$.
Similarly, we can define classes
$\widetilde{S}^{\infty}_{\rho,\delta}(\overline{\Omega}\times\ind)$
and $\widetilde{S}^{-\infty}(\overline{\Omega}\times\ind)$.
\end{definition}

If $a\in S^m_{\rho,\delta}(\overline{\Omega}\times\ind)$, it is convenient to
denote by $a(X,D)= {\rm Op}_L (a)$  the corresponding ${
L}$-pseudo-differential operator defined by
\begin{equation}\label{EQ: L-tor-pseudo-def}
  {\rm Op}_L(a)f(x)=a(X,D)f(x):=\sum_{\xi\in\ind} u_{\xi}(x)\ a(x,\xi)\widehat{f}(\xi).
\end{equation}
The set of operators $ {\rm Op}_L (a)$ of the form
(\ref{EQ: L-tor-pseudo-def}) with $a\in
S^m_{\rho,\delta}(\overline{\Omega}\times\ind)$ will be denoted by
$ {\rm Op}_L (S^m_{\rho,\delta} (\overline{\Omega}\times\ind))$, or by
$\Psi^m_{\rho,\delta} (\overline{\Omega}\times\ind)$. If an
operator $A$ satisfies $A\in{\rm
Op_L}(S^m_{\rho,\delta}(\overline{\Omega}\times\ind))$, we denote
its ${L}$-symbol by $\sigma_{A}=\sigma_{A}(x, \xi), \,\,
x\in\overline{\Omega}, \, \xi\in\ind$.

\begin{remark}\label{REM: Topology of SymClass}
{\rm (Topology on $S^{m}_{\rho,
\delta}(\overline{\Omega}\times\ind)$ ($\widetilde{S}^{m}_{\rho,
\delta}(\overline{\Omega}\times\ind)$)).} The set $S^{m}_{\rho,
\delta}(\overline{\Omega}\times\ind)$ ($\widetilde{S}^{m}_{\rho,
\delta}(\overline{\Omega}\times\ind)$) of symbols has a natural
topology. Let us consider the functions $p_{\alpha\beta}^{l}:
S^{m}_{\rho,
\delta}(\overline{\Omega}\times\ind)\rightarrow\mathbb R$
($\widetilde{p}_{\alpha\beta}^{l}: \widetilde{S}^{m}_{\rho,
\delta}(\overline{\Omega}\times\ind)\rightarrow\mathbb R$) defined
by
$$
p_{\alpha\beta}^{l}(\sigma):={\rm
sup}\left[{\left|\Delta_{(x)}^{\alpha}D^{(\beta)}\sigma(x,
\xi)\right|}{\langle\xi\rangle^{-l+\rho|\alpha|-\delta|\beta|}}:\,\,
(x, \xi)\in\overline{\Omega}\times\ind\right]
$$
$$
\left(\widetilde{p}_{\alpha\beta}^{l}(\sigma):={\rm
sup}\left[{\left|\widetilde{\Delta}_{(x)}^{\alpha}\widetilde{D}^{(\beta)}\sigma(x,
\xi)\right|}{\langle\xi\rangle^{-l+\rho|\alpha|-\delta|\beta|}}:\,\,
(x, \xi)\in\overline{\Omega}\times\ind\right]\right).
$$
Now $\{p_{\alpha\beta}^{l}\}$
($\{\widetilde{p}_{\alpha\beta}^{l}\}$) is a countable family of seminorms,
and they define a
Fr\'echet topology on $S^{m}_{\rho,
\delta}(\overline{\Omega}\times\ind)$
($\widetilde{S}^{m}_{\rho, \delta}(\overline{\Omega}\times\mathbb
Z)$). 
\end{remark}

The next theorem is a prelude to asymptotic expansions, which are
the main tool in the symbolic analysis of ${
L}$-pseudo-differential operators.

\begin{theorem}[Asymptotic sums of symbols] Let $(m_{j})_{j=0}^{\infty}\subset\mathbb
R$ be a sequence such that $m_{j}>m_{j+1}$, and
$m_{j}\rightarrow-\infty$ as $j\rightarrow\infty$, and
$\sigma_{j}\in
S^{m_{j}}_{\rho,\delta}(\overline{\Omega}\times\ind)$ for all
$j\in\ind$. Then there exists an ${L}$-symbol $\sigma\in
S^{m_{0}}_{\rho,\delta}(\overline{\Omega}\times\ind)$ such that
for all $N\in\ind$,
$$
\sigma\stackrel{m_{N},\rho,\delta}{\sim}\sum_{j=0}^{N-1}\sigma_{j}.
$$
\end{theorem}

We will now look at the formulae in \cite{RT2016} for the symbol of the adjoint operator and for the composition of pseudo-differential operators, which establish the pseudo-differential calculus for boundary value problems from the non-harmonic point of view.

\begin{theorem}[Adjoint operators]
Let $0\leq\delta<\rho\leq 1$. Let $A\in {\rm Op}_L
(S^m_{\rho,\delta}(\overline{\Omega}\times\ind))$.
Assume that the conjugate symbol
class $\widetilde{S}^{m}_{\rho,\delta}(\overline{\Omega}\times\ind)$
is defined with strongly admissible
functions $\widetilde{q}_{j}(x,y):=\overline{q_{j}(x,y)}$ which are ${L}^{*}$-strongly admissible.
Then the adjoint of $A$ satisfies
$A^{\ast}\in {\rm Op}_{L^*}(\widetilde{S}^{m}_{\rho,\delta}(\overline{\Omega}\times\ind))$,
with its ${L}^*$-symbol
$\tau_{A^*}\in \widetilde{S}^{m}_{\rho,\delta}(\overline{\Omega}\times\ind)$
having the asymptotic expansion
$$
\tau_{A^*}(x,\xi) \sim \sum_\alpha \frac{1}{\alpha!}
\widetilde \Delta_x^\alpha D_x^{(\alpha)}\overline{\sigma_A(x,\xi)}.
$$
\end{theorem}

We now formulate the composition formula given \cite{RT2016}.

\begin{theorem}\label{Composition}
Let $m_{1}, m_{2}\in\mathbb R$ and $\rho>\delta\geq0$. Let $A,
B:C_{{L}}^{\infty}(\overline{\Omega})\rightarrow C_{{
L}}^{\infty}(\overline{\Omega})$ be continuous and linear, and assume that
their {L}-symbols satisfy
\begin{align*}
|\Delta_{(x)}^{\alpha}\sigma_{A}(x,\xi)|&\leq
C_{\alpha}\langle\xi\rangle^{m_{1}-\rho|\alpha|},\\
|D^{(\beta)}\sigma_{B}(x,\xi)|&\leq
C_{\beta}\langle\xi\rangle^{m_{2}+\delta|\beta|},
\end{align*}
for all $\alpha,\beta\geq 0$, uniformly in $x\in\overline{\Omega}$ and
$\xi\in\ind$.
Then
\begin{equation}
\sigma_{AB}(x,\xi)\sim\sum_{\alpha\geq 0}
\frac{1}{\alpha!}(\Delta_{(x)}^{\alpha}\sigma_{A}(x,\xi))D^{(\alpha)}\sigma_{B}(x,\xi),
\label{CompositionForm}
\end{equation}
where the asymptotic expansion means that for every $N\in\mathbb N$ we have
$$
|\sigma_{AB}(x,\xi)-\sum_{|\alpha|<N}\frac{1}{\alpha!}(\Delta_{(x)}^{\alpha}\sigma_{A}(x,\xi))D^{(\alpha)}\sigma_{B}(x,\xi)|\leq
C_{N}\langle\xi\rangle^{m_{1}+m_{2}-(\rho-\delta)N}.
$$
\end{theorem}

\subsection{Construction of parametrices} Now, we will present a technical result about the existence of parametrices for $L$-elliptic operators  in the global pseudo-differential calculus from \cite{RT2016}. We denote $ {S}^{-\infty}(M\times \mathcal{I})=\cap_{m\in \mathbb{R}}{S}^{m}_{\rho,\delta}(M\times \mathcal{I})=\cap_{m\in \mathbb{R}}{S}^{m}_{1,0}(M\times \mathcal{I}).$
\begin{proposition}\label{IesTParametrix} Let $m\in \mathbb{R},$ and let $0\leqslant \delta<\rho\leqslant 1.$  Let  $a=a(x,\xi)\in {S}^{m}_{\rho,\delta}(M\times \mathcal{I}).$  Assume also that $a(x,\xi)$ is invertible for every $(x,\xi)\in M\times\mathcal{I}
,$ and satisfies
\begin{equation}\label{Iesparametrix}
   \sup_{(x,\xi)\in M\times\mathcal{I}} \vert \langle \xi\rangle^ma(x,\xi)^{-1}\vert<\infty.
\end{equation}Then, there exists $B\in {S}^{-m}_{\rho,\delta}(M\times \mathcal{I}),$ such that $AB-I,BA-I\in {S}^{-\infty}(M\times \mathcal{I}). $ Moreover, the symbol of $B$ satisfies the following asymptotic expansion
\begin{equation}\label{AE}
    \widehat{B}(x,\xi)\sim \sum_{N=0}^\infty\widehat{B}_{N}(x,\xi),\,\,\,(x,\xi)\in M\times \mathcal{I},
\end{equation}where $\widehat{B}_{N}(x,\xi)\in {S}^{-m-(\rho-\delta)N}_{\rho,\delta}(M\times \mathcal{I})$ obeys to the inductive  formula
\begin{equation}\label{conditionelip}
    \widehat{B}_{N}(x,\xi)=-a(x,\xi)^{-1}\left(\sum_{k=0}^{N-1}\sum_{|\gamma|=N-k}(\Delta_{(x)}^{\gamma} a(x,\xi))(D_{x}^{(\gamma)}\widehat{B}_{k}(x,\xi))\right),\,\,N\geqslant 1,
\end{equation}with $ \widehat{B}_{0}(x,\xi)=a(x,\xi)^{-1}.$
\end{proposition}

\section{Parameter $L$-ellipticity} \label{Parameterellipticity}
We start our contributions to the pseudo-differential calculus in the context of non-harmonic analysis developed by the last two authors in \cite{RT2016},
by developing the  functional calculus for H\"ormander classes $\textnormal{Op}(S^{m}_{\rho,\delta}(M\times \mathcal{I}))$. For this,  we need a more wide notion of ellipticity, which we introduce as follows.  
\begin{definition} Let $m>0,$ and let $0\leqslant \delta<\rho\leqslant 1.$
Let $\Lambda=\{\gamma(t):t\in I\footnote{where $I=[a,b],$ $-\infty<a\leqslant b<\infty,$ $I=[a,\infty),$ $I=(-\infty,b]$ or $I=(-\infty,\infty).$}\}$  be an analytic curve in the complex plane $\mathbb{C}.$ If $I$ is a finite interval we assume that $\Lambda$ is a closed curve. For  simplicity, if $I$ is an infinite interval we assume that $\Lambda$ is homotopy equivalent to the line $\Lambda_{i\mathbb{R}}:=\{iy:-\infty<y<\infty\}.$  Let  $a=a(x,\xi)\in {S}^{m}_{\rho,\delta}(M\times \mathcal{I}).$  Assume also that $R_{\lambda}(x,\xi)^{-1}:=a(x,\xi)-\lambda \neq 0$ for every $(x,\xi)\in M\times\mathcal{I},$ and $\lambda \in  \Lambda.$ We say that $a$ is parameter $L$-elliptic with respect to $\Lambda,$ if
\begin{equation*}
    \sup_{\lambda\in \Lambda}\sup_{(x,\xi)\in M\times\mathcal{I}}\vert (|\lambda|^{\frac{1}{m}}+\langle \xi\rangle)^{m}R_{\lambda}(x,\xi)\vert<\infty.
\end{equation*}
\end{definition}

The following theorem classifies the  resolvent $R_{\lambda}(x,\xi)$ of a parameter $L$-elliptic symbol $a.$
\begin{theorem}\label{lambdalambdita} Let $m>0,$ and let $0\leqslant \delta<\rho\leqslant 1.$ If $a$ is parameter $L$-elliptic with respect to $\Lambda,$ the following estimate 
\begin{equation*}
   \sup_{\lambda\in \Lambda}\sup_{(x,\xi)\in M\times\mathcal{I}}\vert (|\lambda|^{\frac{1}{m}}+\langle \xi\rangle)^{m(k+1)}\langle \xi\rangle^{\rho|\alpha|-\delta|\beta|}\partial_{\lambda}^kD_{x}^{(\beta)}\Delta_{(x)}^{\alpha}R_{\lambda}(x,\xi)\vert<\infty,
\end{equation*}holds true for all $\alpha,\beta\in \mathbb{N}_0^n$ and $k\in \mathbb{N}_0.$

\end{theorem}
\begin{proof}
    We will split the proof in the cases $|\lambda|\leqslant 1,$  and $|\lambda|> 1,$ where $\lambda\in \Lambda.$ It is possible however that one of these two cases could be trivial in the sense that $\Lambda_{1}:=\{\lambda\in \Lambda:|\lambda|\leqslant 1\}$ or $\Lambda_{1}^{c}:=\{\lambda\in \Lambda:|\lambda|> 1\}$ could be an empty set. In such a case the proof is self-contained in the situation that we will consider where we assume that $\Lambda_{1}$ and $\Lambda_{1}^c$ are not trivial sets.   For $|\lambda|\leqslant 1,$ observe that
\begin{align*}
    &\vert (|\lambda|^{\frac{1}{m}}+\langle \xi\rangle)^{m(k+1)}\langle \xi\rangle^{\rho|\alpha|-\delta|\beta|}\partial_{\lambda}^kD_{x}^{(\beta)}\Delta_{(x)}^{\alpha}R_{\lambda}(x,\xi)\vert\\
    &=\vert (|\lambda|^{\frac{1}{m}}+\langle \xi\rangle)^{m(k+1)}\langle \xi\rangle^{-m(k+1)}\langle \xi\rangle^{m(k+1)+\rho|\alpha|-\delta|\beta|}\partial_{\lambda}^kD_{x}^{(\beta)}\Delta_{(x)}^{\alpha}R_{\lambda}(x,\xi)\vert\\
    &\leqslant \vert (|\lambda|^{\frac{1}{m}}+\langle \xi\rangle)^{m(k+1)}\langle \xi\rangle^{-m(k+1)}\vert\\
    &\hspace{5cm}\times\vert\langle \xi\rangle^{m(k+1)+\rho|\alpha|-\delta|\beta|}\partial_{\lambda}^kD_{x}^{(\beta)}\Delta_{(x)}^{\alpha}R_{\lambda}(x,\xi)\vert.
\end{align*} We note that 
\begin{align*}
    &\vert (|\lambda|^{\frac{1}{m}}+\langle \xi\rangle)^{m(k+1)}\langle \xi\rangle^{-m(k+1)}\vert\\
    &= \vert (|\lambda|^{\frac{1}{m}}\langle \xi\rangle^{-1}+1)^{m(k+1)}\vert= \vert |\lambda|^{\frac{1}{m}}\langle \xi\rangle^{-1}+1\vert^{m(k+1)}\\
    &\leqslant \sup_{|\lambda|\in [0,1]}\vert (|\lambda|^{\frac{1}{m}}\langle \xi\rangle^{-1}+1)^{m(k+1)}\vert=O(1).
\end{align*} On the other hand, we can prove that
\begin{align*}
   \vert\langle \xi\rangle^{m(k+1)+\rho|\alpha|-\delta|\beta|}\partial_{\lambda}^kD_{x}^{(\beta)}\Delta_{(x)}^{\alpha}R_{\lambda}(x,\xi)\vert=O(1).
\end{align*} For $k=1,$ $\partial_{\lambda}R_{\lambda}(x,\xi)= R_{\lambda}(x,\xi)^{2}.$ This can be deduced from the Leibniz rule, indeed,
\begin{align*}
 0=\partial_{\lambda}(R_{\lambda}(x,\xi)(a(x,\xi)-\lambda))=(\partial_{\lambda}R_{\lambda}(x,\xi)) (a(x,\xi)-\lambda)+ R_{\lambda}(x,\xi)(-1) 
\end{align*}implies that
    \begin{align*}
 -\partial_{\lambda}(R_{\lambda}(x,\xi))(a(x,\xi)-\lambda)=- R_{\lambda}(x,\xi). 
\end{align*} Because $(a(x,\xi)-\lambda)=R_{\lambda}(x,\xi)^{-1}$ the identity for the first derivative of $R_\lambda,$ $\partial_{\lambda}R_{\lambda}$ it follows. So, from the chain rule we obtain that the term of higher order expanding the derivative   $ \partial_{\lambda}^kR_{\lambda} $ is a multiple of $ R_{\lambda}^{k+1}.$ So, $R_{\lambda}\in S^{-m}_{\rho,\delta}(M\times \mathcal{I}).$ The global pseudo-differential calculus implies that $R_{\lambda}^{k+1}\in S^{-m(k+1)}_{\rho,\delta}(M\times \mathcal{I}).$ This fact, and the compactness of $\Lambda_1\subset \mathbb{C},$ provide us the uniform estimate 
    \begin{equation*}
   \sup_{\lambda\in \Lambda_1}\sup_{(x,\xi)\in M\times\mathcal{I}}\vert\langle \xi\rangle^{m(k+1)+\rho|\alpha|-\delta|\beta|}\partial_{\lambda}^kD_{x}^{(\beta)}\Delta_{(x)}^{\alpha}R_{\lambda}(x,\xi)\vert<\infty.
\end{equation*}Now, we will analyse the situation for $\lambda\in \Lambda_1^c.$ We will use induction over $k$ in order to prove that
\begin{equation*}
   \sup_{\lambda\in \Lambda_1^c}\sup_{(x,\xi)\in M\times\mathcal{I}}\vert (|\lambda|^{\frac{1}{m}}+\langle \xi\rangle)^{m(k+1)}\langle \xi\rangle^{\rho|\alpha|-\delta|\beta|}\partial_{\lambda}^kD_{x}^{(\beta)}\Delta_{(x)}^{\alpha}R_{\lambda}(x,\xi)\vert<\infty.
\end{equation*}
For $k=0$ notice that
\begin{align*}
     &\vert (|\lambda|^{\frac{1}{m}}+\langle \xi\rangle)^{m(k+1)}\langle \xi\rangle^{\rho|\alpha|-\delta|\beta|}\partial_{\lambda}^kD_{x}^{(\beta)}\Delta_{(x)}^{\alpha}R_{\lambda}(x,\xi)\vert\\
     &=\vert (|\lambda|^{\frac{1}{m}}+\langle \xi\rangle)^{m}\langle \xi\rangle^{\rho|\alpha|-\delta|\beta|}D_{x}^{(\beta)}\Delta_{(x)}^{\alpha}(a(x,\xi)-\lambda)^{-1}\vert,
\end{align*}and denoting $\theta=\frac{1}{|\lambda|},$ $\omega=\frac{\lambda}{|\lambda|},$ we have
 \begin{align*}
     &\vert (|\lambda|^{\frac{1}{m}}+\langle \xi\rangle)^{m(k+1)}\langle \xi\rangle^{\rho|\alpha|-\delta|\beta|}\partial_{\lambda}^kD_{x}^{(\beta)}\Delta_{(x)}^{\alpha}R_{\lambda}(x,\xi)\vert\\
     &=\vert (|\lambda|^{\frac{1}{m}}+\langle \xi\rangle^{m})|\lambda|^{-1}\langle \xi\rangle^{\rho|\alpha|-\delta|\beta|}D_{x}^{(\beta)}\Delta_{(x)}^{\alpha}(\theta\times  a(x,\xi)-\omega)^{-1}\vert\\
     &=\vert (1+|\lambda|^{-\frac{1}{m}}\langle \xi\rangle)^{m}\langle \xi\rangle^{\rho|\alpha|-\delta|\beta|}D_{x}^{(\beta)}\Delta_{(x)}^{\alpha}(\theta\times  a(x,\xi)-\omega)^{-1}\vert\\
     &=\vert (1+\theta^{\frac{1}{m}}\langle \xi\rangle)^{m}\langle \xi\rangle^{\rho|\alpha|-\delta|\beta|}D_{x}^{(\beta)}\Delta_{(x)}^{\alpha}(\theta\times  a(x,\xi)-\omega)^{-1}\vert\\
     &=\vert (1+\theta^{\frac{1}{m}}\langle \xi\rangle)^{m}\langle \xi\rangle^{-m}\langle \xi\rangle^{m+\rho|\alpha|-\delta|\beta|}D_{x}^{(\beta)}\Delta_{(x)}^{\alpha}(\theta\times  a(x,\xi)-\omega)^{-1}\vert\\
     &\leqslant \vert (1+\theta^{\frac{1}{m}}\langle \xi\rangle)^{m}\langle \xi\rangle^{-m}\vert\vert\langle \xi\rangle^{m+\rho|\alpha|-\delta|\beta|}D_{x}^{(\beta)}\Delta_{(x)}^{\alpha}(\theta\times  a(x,\xi)-\omega)^{-1}\vert.
\end{align*}   Observe that $ (1+\theta^{\frac{1}{m}}\langle \xi\rangle)^{m}\langle \xi\rangle^{-m} \in S^{0}_{\rho,\delta}(M\times \mathcal{I}) ,$  is uniformly bounded in $\theta\in [0,1].$ Similarly, observe that    $$\sup_{\theta\in [0,1]} \vert\langle \xi\rangle^{m+\rho|\alpha|-\delta|\beta|}D_{x}^{(\beta)}\Delta_{(x)}^{\alpha}(\theta\times  a(x,\xi)-\omega)^{-1}\vert<\infty. $$
   Indeed, $(\theta\times  a(x,\xi)-\omega)^{-1}\in S^{-m}_{\rho,\delta}(M\times \mathcal{I}), $ with $\theta\in [0,1]$ and $\omega$ being an element of the complex circle. The case  $k\geqslant 1$ for $\lambda\in \Lambda_1^c$ can be proved in an analogous way.
\end{proof} Combining Proposition \ref{IesTParametrix} and Theorem \ref{lambdalambdita} we obtain the following corollaries.
\begin{corollary}\label{parameterparametrix}
Let $m>0,$ and let $0\leqslant \delta<\rho\leqslant 1.$ Let  $a$ be a parameter $L$-elliptic symbol with respect to $\Lambda.$ Then  there exists a parameter-dependent parametrix of $A-\lambda I,$ with symbol $a^{-\#}(x,\xi,\lambda)$ satisfying the estimates
\begin{equation*}
   \sup_{\lambda\in \Lambda}\sup_{(x,\xi)\in M\times\mathcal{I}}\vert (|\lambda|^{\frac{1}{m}}+\langle \xi\rangle)^{m(k+1)}\langle \xi\rangle^{\rho|\alpha|-\delta|\beta|}\partial_{\lambda}^kD_{x}^{(\beta)}\Delta_{(x)}^{\alpha}a^{-\#}(x,\xi,\lambda)\vert<\infty,
\end{equation*}for all $\alpha,\beta\in \mathbb{N}_0^n$ and $k\in \mathbb{N}_0.$
\end{corollary}
\begin{corollary}\label{resolv}
Let $m>0,$ and let $a\in S^m_{\rho,\delta}(M\times \mathcal{I}) $ where  $0\leqslant \delta<\rho\leqslant 1.$ Let us assume that $\Lambda$ is a subset of the $L^2$-resolvent set of $A,$ $\textnormal{Resolv}(A):=\mathbb{C}\setminus \textnormal{Spec}(A).$ Then $A-\lambda I$ is invertible on $\mathcal{D}'_L(M)$ and the symbol of the resolvent operator $\mathcal{R}_{\lambda}:=(A-\lambda I)^{-1},$ $\widehat{\mathcal{R}}_{\lambda}(x,\xi)$ belongs to $S^{-m}_{\rho,\delta}(M\times \mathcal{I}).$ 
\end{corollary}

\section{Global functional calculus}\label{SFC}
In this section we develop the global functional calculus for the classes $S^{m}_{\rho,\delta}(M\times \mathcal{I}).$ The global pseudo-differential calculus will be applied to obtain a global G\r{a}rding inequality.  

\subsection{Symbols defined by functions of pseudo-differential operators}\label{S8}
Let $a\in S^m_{\rho,\delta}(M\times \mathcal{I})$ be a parameter $L$-elliptic symbol  of order $m>0$ with respect to the sector $\Lambda\subset\mathbb{C}.$ For $A=\textnormal{Op}(a),$ let us define the operator $F(A)$  by the (Dunford-Riesz) complex functional calculus
\begin{equation}\label{F(A)}
    F(A)=-\frac{1}{2\pi i}\oint\limits_{\partial \Lambda_\varepsilon}F(z)(A-zI)^{-1}dz,
\end{equation}where
\begin{itemize}
    \item[(CI).] $\Lambda_{\varepsilon}:=\Lambda\cup \{z:|z|\leqslant \varepsilon\},$ $\varepsilon>0,$ and $\Gamma=\partial \Lambda_\varepsilon\subset\textnormal{Resolv}(A)$ is a positively oriented curve in the complex plane $\mathbb{C}$.
    \item[(CII).] $F$ is a holomorphic function in $\mathbb{C}\setminus \Lambda_{\varepsilon},$ and continuous on its closure. 
    \item[(CIII).] We will assume  decay of $F$ along $\partial \Lambda_\varepsilon$ in order that the operator \eqref{F(A)} will be densely defined on $C^\infty_L(M)$ in the strong sense of the topology on $L^2(M).$
\end{itemize} Now, we will compute the global symbols for operators defined by this complex functional calculus. So, we will assume the $\textnormal{WZ}$ condition.
\begin{lemma}\label{LemmaFC}
Let $a\in S^m_{\rho,\delta}(M\times \mathcal{I})$ be a parameter $L$-elliptic symbol  of order $m>0$ with respect to the sector $\Lambda\subset\mathbb{C}.$ Let $F(A):C^\infty_L(M)\rightarrow \mathcal{D}'_L(M)$ be the operator defined by the analytical functional calculus as in \eqref{F(A)}. Under the assumptions $\textnormal{(CI)}$, $\textnormal{(CII)}$, and $\textnormal{(CIII)}$, the global symbol of $F(A),$ $\sigma_{F(A)}(x,\xi)$ is given by,
\begin{equation*}
    \sigma_{F(A)}(x,\xi)=-\frac{1}{2\pi i}\oint\limits_{\partial \Lambda_\varepsilon}F(z)\widehat{\mathcal{R}}_z(x,\xi)dz,
\end{equation*}where $\mathcal{R}_z=(A-zI)^{-1}$ denotes the resolvent of $A,$ and $\widehat{\mathcal{R}}_z(x,\xi)\in S^{-m}_{\rho,\delta}(M\times \mathcal{I}) $ its symbol.
\end{lemma}
\begin{proof}
 From Corollary \ref{resolv}, we have that  $\widehat{\mathcal{R}}_z(x,\xi)\in S^{-m}_{\rho,\delta}(M\times \mathcal{I}) .$ Now, observe that  \begin{align*}
  \sigma_{F(A)}(x,\xi)=u_\xi(x)^{-1}F(A)u_\xi(x)=-\frac{1}{2\pi i}\oint\limits_{\partial \Lambda_\varepsilon}F(z)u_\xi(x)^{-1}(A-zI)^{-1}u_\xi(x)dz.  \end{align*} We finish the proof by observing that $\widehat{\mathcal{R}}_z(x,\xi)=u_\xi(x)^{-1}(A-zI)^{-1}u_\xi(x),$ for every $z\in \textnormal{Resolv}(A).$
\end{proof}
Assumption (CIII) will be clarified in the following theorem where we show that the global pseudo-differential calculus is stable under the action of the complex functional calculus.
\begin{theorem}\label{DunforRiesz}
Let $m>0,$ and let $0\leqslant \delta<\rho\leqslant 1.$ Let  $a\in S^m_{\rho,\delta}(M\times \mathcal{I})$ be a parameter $L$-elliptic symbol with respect to $\Lambda.$ Let us assume that $F$ satisfies the  estimate $|F(\lambda)|\leqslant C|\lambda|^s$ uniformly in $\lambda,$ for some $s<0.$  Then  the symbol of $F(A),$  $\sigma_{F(A)}\in S^{ms}_{\rho,\delta}(M\times \mathcal{I}) $ admits an asymptotic expansion of the form
\begin{equation}\label{asymcomplex}
    \sigma_{F(A)}(x,\xi)\sim 
     \sum_{N=0}^\infty\sigma_{{B}_{N}}(x,\xi),\,\,\,(x,\xi)\in M\times \mathcal{I},
\end{equation}where $\sigma_{{B}_{N}}(x,\xi)\in {S}^{ms-(\rho-\delta)N}_{\rho,\delta}(M\times \mathcal{I})$ and 
\begin{equation*}
    \sigma_{{B}_{0}}(x,\xi)=-\frac{1}{2\pi i}\oint\limits_{\partial \Lambda_\varepsilon}F(z)(a(x,\xi)-z)^{-1}dz\in {S}^{ms}_{\rho,\delta}(M\times \mathcal{I}).
\end{equation*}Moreover, 
\begin{equation*}
     \sigma_{F(A)}(x,\xi)\equiv -\frac{1}{2\pi i}\oint\limits_{\partial \Lambda_\varepsilon}F(z)a^{-\#}(x,\xi,\lambda)dz \textnormal{  mod  } {S}^{-\infty}(M\times \mathcal{I}),
\end{equation*}where $a^{-\#}(x,\xi,\lambda)$ is the symbol of the parametrix to $A-\lambda I,$   in Corollary \ref{parameterparametrix}.
\end{theorem}
\begin{proof}
    First, we need to prove that the condition $|F(\lambda)|\leqslant C|\lambda|^s$ uniformly in $\lambda,$ for some $s<0,$ is enough in order to guarantee that \begin{equation*}
    \sigma_{{B}_{0}}(x,\xi):=-\frac{1}{2\pi i}\oint\limits_{\partial \Lambda_\varepsilon}F(z)(a(x,\xi)-z)^{-1}dz,
\end{equation*} is a well defined global-symbol.
From Theorem \ref{lambdalambdita} we deduce that $(a(x,\xi)-z)^{-1}$ satisfies the estimate
\begin{equation*}
   \vert (|z|^{\frac{1}{m}}+\langle \xi\rangle)^{m(k+1)}\langle \xi\rangle^{\rho|\alpha|-\delta|\beta|}\partial_{z}^kD_{x}^{(\beta)}\Delta_{(x)}^{\alpha}(a(x,\xi)-z)^{-1}\vert<\infty.
\end{equation*}
Observe that 
\begin{align*}
    &\vert(a(x,\xi)-z)^{-1}\vert\\
    & =\vert (|z|^{\frac{1}{m}}+\langle \xi\rangle)^{-m}(|z|^{\frac{1}{m}}+\langle \xi\rangle)^{m}(a(x,\xi)-z)^{-1}\vert \\
    &\lesssim  (|z|^{\frac{1}{m}}+\langle \xi\rangle)^{-m}\\&\leqslant |z|^{-1},
\end{align*} and the condition $s<0$ implies
\begin{align*}
    \left|\frac{1}{2\pi i}\oint\limits_{\partial \Lambda_\varepsilon}F(z)(a(x,\xi)-z)^{-1}dz\right|\lesssim \oint\limits_{\partial \Lambda_\varepsilon}|z|^{-1+s}|dz|<\infty,
\end{align*}uniformly in $(x,\xi)\in M\times \mathcal{I}.$ In order to check that $\sigma_{B_0}\in {S}^{ms}_{\rho,\delta}(M\times \mathcal{I})$ let us analyse the cases $-1<s<0$ and $s\leqslant -1$ separately: 

{\textbf{Case 1:}} Let us analyse first the situation of $-1<s<0.$ We observe that
\begin{align*}
   &\vert \langle \xi\rangle^{-ms+\rho|\alpha|-\delta|\beta|}D_{x}^{(\beta)}\Delta_{(x)}^{\alpha}\sigma_{B_0}(x,\xi)\vert\\
   &\leqslant \frac{C}{2\pi }\oint\limits_{\partial \Lambda_\varepsilon} |z|^{s}\vert      \langle \xi\rangle^{-ms+\rho|\alpha|-\delta|\beta|}D_{x}^{(\beta)}\Delta_{(x)}^{\alpha}(a(x,\xi)-z)^{-1}\vert |dz|.
\end{align*}Now, we will estimate the operator norm inside of the integral. Indeed, the identity
\begin{align*}
    &\vert      \langle \xi\rangle^{-ms+\rho|\alpha|-\delta|\beta|}D_{x}^{(\beta)}\Delta_{(x)}^{\alpha}(a(x,\xi)-z)^{-1}\vert=\\
    &\vert (|z|^{\frac{1}{m}}+\langle \xi\rangle)^{-m}(|z|^{\frac{1}{m}}+\langle \xi\rangle)^{m}\langle \xi\rangle^{-ms+\rho|\alpha|-\delta|\beta|}D_{x}^{(\beta)}\Delta_{(x)}^{\alpha}(a(x,\xi)-z)^{-1}\vert
\end{align*}implies that
\begin{align*}
    &\vert      \langle \xi\rangle^{-ms+\rho|\alpha|-\delta|\beta|}D_{x}^{(\beta)}\Delta_{(x)}^{\alpha}(a(x,\xi)-z)^{-1}\vert \lesssim  \vert (|z|^{\frac{1}{m}}+\langle \xi\rangle)^{-m}\langle \xi\rangle^{-ms}\vert
\end{align*}where we have used that
\begin{align*}
\sup_{z\in \partial \Lambda_\varepsilon}  \sup_{(x,\xi)} \vert (|z|^{\frac{1}{m}}+\langle \xi\rangle)^{m}\langle \xi\rangle^{\rho|\alpha|-\delta|\beta|}D_{x}^{(\beta)}\Delta_{(x)}^{\alpha}(a(x,\xi)-z)^{-1}\vert <\infty.
\end{align*}
Consequently, by using that  $s<0,$ we deduce
\begin{align*}
   & \frac{C}{2\pi }\oint\limits_{\partial \Lambda_\varepsilon} |z|^{s}\vert      \langle \xi\rangle^{ms+\rho|\alpha|-\delta|\beta|}D_{x}^{(\beta)}\Delta_{(x)}^{\alpha}(a(x,\xi)-z)^{-1}\vert |dz|\\
    &\lesssim \frac{C}{2\pi }\oint\limits_{\partial \Lambda_\varepsilon} |z|^{s}\vert (|z|^{\frac{1}{m}}+\langle \xi\rangle)^{-m} \langle \xi\rangle^{-ms}\vert |dz|
    .
\end{align*}To study the convergence of the last contour integral we only need to check the convergence of $\int_{1}^{\infty}r^s(r^{\frac{1}{m}}+\varkappa)^{-m}\varkappa^{-ms}dr,$ where $\varkappa>1$ is a parameter. The change of variable $r=\varkappa^{m}t$ implies that
\begin{align*}
   \int\limits_{1}^{\infty}r^s(r^{\frac{1}{m}}+\varkappa)^{-m}\varkappa^{-ms}dr&=\int\limits_{\varkappa^{-m}}^{\infty}\varkappa^{ms}t^s(\varkappa t^{\frac{1}{m}}+\varkappa)^{-m}\varkappa^{-ms}\varkappa^mdt=\int\limits_{\varkappa^{-m}}^{\infty}t^s(t^{\frac{1}{m}}+1)^{-m}dt\\
   &\lesssim \int\limits_{\varkappa^{-m}}^{1}t^sdt+\int\limits_{1}^{\infty}t^{-1+s}<\infty.
\end{align*}Indeed, for $t\rightarrow\infty,$ $t^s(t^{\frac{1}{m}}+1)^{-m}\lesssim t^{-1+s}$ and we conclude the estimate because $\int\limits_{1}^{\infty} t^{-1+s'}dt<\infty,$ for all $s'<0.$ On the other hand, the condition $-1<s<0$ implies that
\begin{align*}
 \int\limits_{\varkappa^{-m}}^{1}t^sdt=\frac{1}{1+s}-\frac{\varkappa^{-m(1+s)}}{1+s}=   O(1).
\end{align*} 

{\textbf{Case 2.}}
In the case where $s\leqslant -1,$ we can find an analytic function $\tilde{G}(z)$ such that it is a holomorphic function in $\mathbb{C}\setminus \Lambda_{\varepsilon},$ and continuous on its closure and additionally satisfying that $F(\lambda)=\tilde{G}(\lambda)^{1+[-s]}.$\footnote{ $[-s]$ denotes  the integer part of $-s.$} In this case,  $\tilde{G}(A)$ defined by the complex functional calculus 
\begin{equation}\label{G(A)}
    \tilde{G}(A)=-\frac{1}{2\pi i}\oint\limits_{\partial \Lambda_\varepsilon}\tilde{G}(z)(A-zI)^{-1}dz,
\end{equation}
has symbol belonging to ${S}^{\frac{sm}{1+[-s]}}_{\rho,\delta}(M\times \mathcal{I}),$ this in view of {\textbf{Case 1}}, because $\tilde{G}$ satisfies the estimate $|\tilde{G}(\lambda)|\leqslant C|\lambda|^{\frac{s}{1+[-s]}},$ with $-1<\frac{s}{1+[-s]}<0.$ 
By observing that
\begin{align*}
    \sigma_{F(A)}(x,\xi)&=-\frac{1}{2\pi i}\oint\limits_{\partial \Lambda_\varepsilon}F(z)\widehat{\mathcal{R}}_z(x,\xi)dz=-\frac{1}{2\pi i}\oint\limits_{\partial \Lambda_\varepsilon}\tilde{G}(z)^{1+[-s]}\widehat{\mathcal{R}}_z(x,\xi)dz\\
    &=\sigma_{\tilde{G}(A)^{1+[-s]}}(x,\xi),
\end{align*}and computing the symbol $\sigma_{\tilde{G}(A)^{1+[-s]}}(x,\xi)$ by iterating $1+[-s]$-times  the asymptotic formula for the composition in the global pseudo-differential calculus, we can see that the term with higher order in such expansion is $\sigma_{\tilde{G}(A)}(x,\xi)^{1+[-s]}\in {S}^{ms}_{\rho,\delta}(M\times \mathcal{I}).$ Consequently we have proved that $\sigma_{F(A)}(x,\xi)\in {S}^{ms}_{\rho,\delta}(M\times \mathcal{I}).$
This completes the proof for the first part of the theorem.
For the second part of the proof, let us denote by $a^{-\#}(x,\xi,\lambda)$  the symbol of the parametrix to $A-\lambda I,$   in Corollary \ref{parameterparametrix}. Let $P_{\lambda}=\textnormal{Op}(a^{-\#}(\cdot,\cdot,\lambda)).$ Because $\lambda\in \textnormal{Resolv}(A)$ for $\lambda\in \partial \Lambda_\varepsilon,$ $(A-\lambda)^{-1}-P_{\lambda}$ is a smoothing operator. Consequently, from Lemma \ref{LemmaFC} we deduce that
\begin{align*}
   & \sigma_{F(A)}(x,\xi)\\
   &=-\frac{1}{2\pi i}\oint\limits_{\partial \Lambda_\varepsilon}F(z)\widehat{\mathcal{R}}_z(x,\xi)dz\\
    &=-\frac{1}{2\pi i}\oint\limits_{\partial \Lambda_\varepsilon}F(z)a^{-\#}(x,\xi,z)dz-\frac{1}{2\pi i}\oint\limits_{\partial \Lambda_\varepsilon}F(z)(\widehat{\mathcal{R}}_z(x,\xi)-a^{-\#}(x,\xi,z))dz\\
    &\equiv -\frac{1}{2\pi i}\oint\limits_{\partial \Lambda_\varepsilon}F(z)a^{-\#}(x,\xi,z)dz  \textnormal{  mod  } {S}^{-\infty}(M\times \mathcal{I}).
\end{align*}The asymptotic expansion \eqref{asymcomplex} came from the construction of the parametrix in the global pseudo-differential calculus (see Proposition \ref{IesTParametrix}).
\end{proof}

\subsection{G\r{a}rding inequality} 
In this subsection we prove the G\r{a}rding inequality for the global pseudo-differential calculus. To do so, we need some preliminaries. 
\begin{proposition}
Let $0\leqslant \delta<\rho\leqslant 1.$  Let  $a\in S^m_{\rho,\delta}(M\times \mathcal{I})$ be an $L$-elliptic global symbol where $m\geqslant 0$ and let us assume that $a>0.$  Then $a$ is parameter-elliptic with respect to $\mathbb{R}_{-}:=\{z=x+i0:x<0\}\subset\mathbb{C}.$ Furthermore, for any number $s\in \mathbb{C},$ 
\begin{equation*}
    \widehat{B}(x,\xi)\equiv a(x,\xi)^s:=\exp(s\log(a(x,\xi))),\,\,(x,\xi)\in M\times \mathcal{I},
\end{equation*}defines a symbol $\widehat{B}(x,\xi)\in S^{m\times\textnormal{Re}(s)}_{\rho,\delta}(M\times \mathcal{I}).$
\end{proposition}
\begin{proof} From the estimates 
\begin{equation*}
     \sup_{(x,\xi)}\vert\langle \xi\rangle^{-m} a(x,\xi) \vert<\infty,\quad \sup_{(x,\xi)}\vert \langle \xi\rangle^{m}a(x,\xi)^{-1} \vert<\infty,
\end{equation*}
   we deduce that
    \begin{equation*}
          \langle \xi\rangle^{-m}|a(x,\xi)|\subset [c,C],
    \end{equation*}where $c,C>0$ are positive real numbers. Now, for every $\lambda\in \mathbb{R}_{-}$ we have
    \begin{align*}
      &   \vert (|\lambda|^{\frac{1}{m}}+\langle \xi\rangle)^m (a(x,\xi)-\lambda)^{-1}\vert\\
         &\asymp \vert (|\lambda|^{\frac{1}{m}}+\langle \xi\rangle)^m (\langle \xi\rangle^m-\lambda)^{-1}(\langle \xi\rangle^m-\lambda)(a(x,\xi)-\lambda)^{-1}\vert \\
         &\lesssim \vert (|\lambda|^{\frac{1}{m}}+\langle \xi\rangle)^m (\langle \xi\rangle^m-\lambda)^{-1}\vert \vert (\langle \xi\rangle^m-\lambda)(a(x,\xi)-\lambda)^{-1}\vert\\
         &\lesssim \vert (|\lambda|^{\frac{1}{m}}+\langle \xi\rangle)^m (\langle \xi\rangle^m-\lambda)^{-1}\vert.
    \end{align*}
By fixing again $\lambda\in \mathbb{R}_{-}$, we observe that from the compactness of $[0,1/2]$ we deduce that
    \begin{align*}
        \sup_{0\leqslant \lambda\leqslant 1/2}\vert (|\lambda|^{\frac{1}{m}}+\langle \xi\rangle)^m (\langle \xi\rangle^m-\lambda)^{-1}\vert&\asymp \sup_{0\leqslant \lambda\leqslant 1/2}\vert \langle \xi\rangle^m (\langle \xi\rangle^m-\lambda)^{-1}\vert \\
        &\asymp \sup_{0\leqslant \lambda\leqslant 1/2}\vert \langle \xi\rangle^m (\langle \xi\rangle^m-\lambda)^{-1}\vert \\
         &\asymp \sup_{0\leqslant \lambda\leqslant 1/2}\vert  (1-\lambda\langle \xi\rangle^{-m})^{-1}\vert \\
          &\lesssim 1.
    \end{align*} On the other hand,
    \begin{align*}
    &  \sup_{\lambda\geqslant 1/2}\vert (|\lambda|^{\frac{1}{m}}+\langle \xi\rangle)^m (\langle \xi\rangle^m-\lambda)^{-1}\vert  \\
      &=\sup_{\lambda\geqslant 1/2}\vert (|\lambda|^{\frac{1}{m}}\langle \xi\rangle^{-1}+1)^m (1-\langle \xi\rangle^{-m}\lambda)^{-1}\vert\\
      &=\sup_{\lambda\geqslant 1/2}\vert (\langle \xi\rangle^{-1}+|\lambda|^{-\frac{1}{m}})^m |\lambda|(1-\langle \xi\rangle^{-m}\lambda)^{-1}\vert\\
      &\lesssim \sup_{\lambda\geqslant 1/2}\vert \langle \xi\rangle^{-m} |\lambda|(-\lambda)^{-1}\langle \xi\rangle^m\vert\\
      &=1.
    \end{align*}So, we have proved that $a$ is parameter-elliptic with respect to $\mathbb{R}_{-}.$ To prove that $\widehat{B}(x,\xi)\in S^{m\times\textnormal{Re}(s)}_{\rho,\delta}(M\times \mathcal{I}),$ we observe that for $\textnormal{Re}(s)<0$, Theorem \ref{DunforRiesz} can be applied. If  $\textnormal{Re}(s)\geqslant 0$ then there exists $k\in \mathbb{N}$ such that $\textnormal{Re}(s)-k<0$. Consequently, from the spectral calculus of matrices we deduce that $a(x,\xi)^{\textnormal{Re}(s)-k}\in S^{m\times(\textnormal{Re}(s)-k)}_{\rho,\delta}(M\times \mathcal{I}).$ So, from the global pseudo-differential calculus we conclude that   $$a(x,\xi)^{s}=a(x,\xi)^{s-k}a(x,\xi)^{k}\in S^{m\times\textnormal{Re}(s)}_{\rho,\delta}(M\times \mathcal{I}).$$ Thus, the proof is complete.
\end{proof}
\begin{corollary}\label{1/2}
Let $0\leqslant \delta,\rho\leqslant 1.$  Let  $a\in S^m_{\rho,\delta}(M\times \mathcal{I}),$  be an $L$-elliptic symbol  where $m\geqslant 0$ and let us assume that $a>0.$  Then 
$\widehat{B}(x,\xi)\equiv a(x,\xi)^\frac{1}{2}:=\exp(\frac{1}{2}\log(a(x,\xi)))\in S^{\frac{m}{2}}_{\rho,\delta}(M\times \mathcal{I}).$
\end{corollary}

Now, we prove the following lower bound.
\begin{theorem}
[G\r{a}rding inequality]
\label{GardinTheorem} 
For $0\leqslant \delta<\rho\leqslant 1,$ let $a(x,D):C^\infty_L(M)\rightarrow\mathcal{D}'_L(M)$ be an operator with symbol  $a\in {S}^{m}_{\rho,\delta}( M\times \mathcal{I})$, $m\in \mathbb{R}$. Let us assume that 
\begin{equation*}
    A(x,\xi):=\frac{1}{2}(a(x,\xi)+\overline{a(x,\xi)}),\,(x,\xi)\in M\times \mathcal{I},\,\,a\in S^m_{\rho,\delta}(M\times \mathcal{I}), 
\end{equation*}satisfies
\begin{align*}\label{garding}
    \vert\langle \xi\rangle^mA(x,\xi)^{-1} \vert\leqslant C_{0}.
\end{align*}Then, there exist $C_{1},C_{2}>0,$ such that the lower bound
\begin{align}
    \textnormal{Re}(a(x,D)u,u) \geqslant C_1\Vert u\Vert_{\mathcal{H}^{\frac{m}{2}}(M)}-C_2\Vert u\Vert_{L^2(M)}^2,
\end{align}holds true for every $u\in C^\infty_L(M).$
\end{theorem}
\begin{proof}
In view of that
\begin{equation*}
    A(x,\xi):=\frac{1}{2}(a(x,\xi)+\overline{a(x,\xi)}),\,(x,\xi)\in M\times \mathcal{I},\,\,a\in S^m_{\rho,\delta}(M\times \mathcal{I}), 
\end{equation*} satisfies
\begin{align}\label{eqi}
    \vert\langle \xi\rangle^mA(x,\xi)^{-1} \vert\leqslant C_{0},
\end{align} 
we get
\begin{align*}
  \langle \xi\rangle^{-m}A(x,\xi)\geqslant\frac{1}{C_0}.
\end{align*}This implies that
\begin{align*}
  A(x,\xi)\geqslant\frac{1}{C_0}\langle \xi\rangle^m,
\end{align*}and for $C_1\in(0, \frac{1}{C_0})$ we have that
\begin{align*}
 A(x,\xi)-C_{1}  \langle \xi\rangle^m\geqslant \left(\frac{1}{C_0}-C_1\right) \langle \xi\rangle^m>0.
\end{align*} 
If  $0\leqslant \delta<\rho\leqslant 1,$ from Corollary \ref{1/2}, we have
\begin{align*}
    q(x,\xi):=(A(x,\xi)-C_{1}  \langle \xi\rangle^m)^{\frac{1}{2}}\in  S^{\frac{m}{2}}_{\rho,\delta}(M\times \mathcal{I}).
\end{align*}From the symbolic calculus we obtain
\begin{align*}
  q(x,D)q(x,D)^{*}= A(x,D)-C_{1}\textnormal{Op}(  \langle \xi\rangle^m)+r(x,D),\,\,   r(x,\xi)\in  S^{m-(\rho-\delta)}_{\rho,\delta}(M\times \mathcal{I}).
\end{align*}
Now, let us assume that $u\in C^\infty_L(M).$ Denoting $\mathcal{M}_{s}:=(1+L^{\circ}L)^{\frac{s}{2\textnormal{ord}(L)}}=\textnormal{Op}(\langle \xi\rangle^s),$ $s\in \mathbb{R},$ we have
\begin{align*}
    \textnormal{Re}(a(x,D)u,u)&=\frac{1}{2}((a(x,D)+\textnormal{Op}(a^*))u,u)=(A(x,D)u,u)\\
    &=C_{1}(\mathcal{M}_{m}u,u)+(q(x,D)q(x,D)^{*}u,u)+(r(x,D)u,u)\\
    &=C_{1}(\mathcal{M}_{m}u,u)+(q(x,D)^{*}u,q(x,D)^*u)-(r(x,D)u,u)\\
    &\geqslant C_{1}\Vert u\Vert_{\mathcal{H}^{\frac{m}{2}}(M)}-(r(x,D)u,u)\\
     &= C_{1}\Vert u\Vert_{\mathcal{H}^{\frac{m}{2}}(M)}-(\mathcal{M}_{-\frac{m-(\rho-\delta)}{2}}r(x,D)u,\mathcal{M}_{\frac{m-(\rho-\delta)}{2}}u).
\end{align*}
Observe that
\begin{align*}
(\mathcal{M}_{-\frac{m-(\rho-\delta)}{2}}r(x,D)u,\mathcal{M}_{\frac{m-(\rho-\delta)}{2}}u)&\leqslant \Vert \mathcal{M}_{-\frac{m-(\rho-\delta)}{2}}r(x,D)u \Vert_{L^2(M)}\Vert u\Vert_{\mathcal{H}^{\frac{m-(\rho-\delta)}{2}}(M)}\\
    &= \Vert r(x,D)u \Vert_{\mathcal{H}^{-\frac{m-(\rho-\delta)}{2}}(M)}\Vert u\Vert_{\mathcal{H}^{\frac{m-(\rho-\delta)}{2}}(M)}\\
     &\leqslant C_1\Vert u \Vert_{\mathcal{H}^{\frac{m-(\rho-\delta)}{2}}(M)}\Vert u\Vert_{\mathcal{H}^{\frac{m-(\rho-\delta)}{2}}(M)},
\end{align*}where in the last line we have used the  Sobolev boundedness of $r(x,D)$ from $\mathcal{H}^{\frac{m-(\rho-\delta)}{2}}(M)$ into $\mathcal{H}^{-\frac{m-(\rho-\delta)}{2}}(M).$ Consequently, we deduce the lower bound
\begin{align*}
    \textnormal{Re}(a(x,D)u,u) \geqslant C_{1}\Vert u\Vert_{\mathcal{H}^{\frac{m}{2}}(M)}-C\Vert u\Vert_{\mathcal{H}^{\frac{m-(\rho-\delta)}{2}}(M)}^2.
\end{align*} If we assume for a moment that for every $\varepsilon>0,$ there exists $C_{\varepsilon}>0,$ such that
\begin{equation}\label{lemararo}
    \Vert u\Vert_{{L}^{2}_{\frac{m-(\rho-\delta)}{2}}(M)}^2\leqslant \varepsilon\Vert u\Vert_{\mathcal{H}^{\frac{m}{2}}(M)}^2+C_{\varepsilon}\Vert u \Vert_{L^2(M)}^2,
\end{equation} for $0<\varepsilon<C_{1}$ we have
\begin{align*}
    \textnormal{Re}(a(x,D)u,u) \geqslant (C_{1}-\varepsilon)\Vert u\Vert^2_{\mathcal{H}^{\frac{m}{2}}(M)}-C_{\varepsilon}\Vert u\Vert_{L^2(M)}^2.
\end{align*}So, with the exception of the proof of \eqref{lemararo}
in view of the analysis above, for the proof of Theorem \ref{GardinTheorem} we only need to prove \eqref{lemararo}. However we will deduce it from the following more general lemma.
\end{proof}
\begin{lemma}Let us assume that $s\geqslant t\geqslant 0$ or that $s,t<0.$ Then, for every $\varepsilon>0,$ there exists $C_\varepsilon>0$ such that 
\begin{equation}\label{lemararo2}
    \Vert u\Vert_{{L}^{2,L}_{t}(M)}^2\leqslant \varepsilon\Vert u\Vert_{{L}^{2,L}_{s}(M)}^2+C_{\varepsilon}\Vert u \Vert_{L^2(M)}^2,
\end{equation}holds true for every $u\in C^\infty_L(M).$ 
\end{lemma}
\begin{proof}
    Let $\varepsilon>0.$ Then, there exists $C_{\varepsilon}>0$ such that
    \begin{equation*}
        \langle \xi\rangle^{2t}-\varepsilon \langle \xi\rangle^{2s}\leqslant C_{\varepsilon},
    \end{equation*}uniformly in $\xi\in \mathcal{I}.$ Then \eqref{lemararo2} it follows from the Plancherel theorem. Indeed,
    \begin{align*}
     \Vert u\Vert_{{L}^{2,L}_{t}(M)}^2& =\sum_{\xi\in \mathcal{I}} \langle \xi\rangle^{2t}|\widehat{u}(\xi)|^{2} \leqslant  \sum_{\xi\in \mathcal{I}} (\varepsilon\langle \xi\rangle^{2s}+C_\varepsilon)|\widehat{u}(\xi)|^{2}  \\
    &= \varepsilon\Vert u\Vert_{{L}^{2,L}_{s}(M)}^2+C_{\varepsilon}\Vert u \Vert_{L^2(M)}^2,
    \end{align*} 
    completing the proof.
\end{proof}
\begin{corollary}\label{GardinTheorem2}  Let $a(x,D):C^\infty_L(M)\rightarrow\mathcal{D}'_L(M)$ be an operator with symbol  $a\in {S}^{m}_{\rho,\delta}( M\times \mathcal{I})$, $m\in \mathbb{R}$. Let us assume that 
\begin{equation*}
    a(x,\xi)\geqslant 0,\,(x,\xi)\in M\times \mathcal{I}, 
\end{equation*}satisfies
\begin{align*}\label{garding22}
    \vert\langle \xi\rangle^ma(x,\xi)^{-1} \vert\leqslant C_{0}.
\end{align*}Then, there exist $C_{1},C_{2}>0,$ such that the lower bound
\begin{align}
    \textnormal{Re}(a(x,D)u,u) \geqslant C_1\Vert u\Vert_{\mathcal{H}^{\frac{m}{2}}(M)}-C_2\Vert u\Vert_{L^2(M)}^2,
\end{align}holds true for every $u\in C^\infty_L(M).$
\end{corollary}

\section{$L^2$-estimates for pseudo-differential operators} 
\label{L^2boun}

In this section we prove the following theorem.
\begin{theorem}
\label{L2}  
Let $a(x,D):C^\infty_L(M)\rightarrow\mathcal{D}'_L(M)$ be a pseudo-differential  operator with symbol  $a\in {S}^{0}_{\rho,\delta}( M\times \mathcal{I})$ with  $0 \leq \delta <\rho \leq 1.$ Then $a(x,D)$ extends to a bounded operator on $L^2({M})$. 
\end{theorem}
\begin{proof}
Assume first that $a(x,\xi)\in {S}^{-m_0}_{\rho',\delta'}(M\times \mathcal{I}),$ where $m_0>0.$ The  kernel of $a(x,D)=\textnormal{Op}(a),$ $K_{a}(x,y),$ belongs to $L^\infty(M\times M)$ for $m_0$ large enough. Indeed, by using
$$
K_{a}(x,y)=\sum_{\xi\in \mathcal{I}}u_{\xi}(x)\overline{v_{\xi}(y)}a(x,\xi),
$$
let us identify for which $m_0,$ $a(x,D)$ is Hilbert-Schmidt. Since, $a(x,D)$ is Hilbert-Schmidt if and only if $K_a\in L^{2}(M\times M)$. By simple calculations, we obtain 
\begin{align*}
\Vert K_a(x,y) \Vert_{L^2(M\times M)} & \leq \sum_{\xi\in \mathcal{I}}\sup_{x\in M}|a(x,\xi)| \Vert u_{\xi}(x) v_{\xi}(y)\Vert_{L^{2}(M\times M)}\\
    &\lesssim  \sum_{\xi\in \mathcal{I}}\langle \xi\rangle^{-m_0} \Vert u_{\xi}\Vert_{L^{2}(M)}\Vert{v_{\xi}}\Vert_{L^{2}(M)}\\
    &=  \sum_{\xi\in \mathcal{I}}\langle \xi\rangle^{-m_0}.
\end{align*} 
Thus, for $m_0\geq s_{0},$ $a(x,D)$ is Hilbert-Schmidt on $L^2(M)$ and, consequently, a bounded operator on $L^2(M)$.

Next by induction we prove that $a(x,D)$ is $L^2$-bounded if $p(x,\xi)\in  {S}^{-m_0,\mathcal{L}}_{\rho',\delta'}(M\times \mathcal{I}),$ for $m_0< m\leqslant -(\rho'-\delta').$ To do so,  for $u\in C^{\infty}(M)$ we form
\begin{align*}
\Vert a(x,D)u \Vert^2_{L^2(M)}&=( a(x,D)u, a(x,D)u)_{L^2(M)}\\
&=( a^{*}(x,D)a(x,D)u, u)_{L^2(M)}\\
&=( b(x,D)u,u )_{L^2(M)} ,
\end{align*}
where $b(x,D)=a^{*}(x,D)a(x,D)$ has a symbol in ${S}^{2m,\mathcal{L}}_{\rho',\delta'}(M\times \mathcal{I}),$ for $0\leqslant \delta'<\rho'\leqslant 1.$ From the induction hypothesis the continuity of $a(x,D)$ for all $a\in S^{2m,\mathcal{L}}_{\rho',\delta'}$ follows successively for $m\leqslant-\frac{m_0}{2},-\frac{m_0}{4},\cdots , -\frac{m_0}{2^{\ell_0}},\cdots ,$ $\ell_0\in\mathbb{N},$ and hence for $m\leqslant -\frac{m_0}{2^{\ell_0}}$ where $\frac{m_0}{2^{\ell_0}}<\rho'-\delta',$ after a finite number of steps. 

Assume that $a(x,\xi)\in S^{0,\mathcal{L}}_{\rho',\delta'}(M\times \mathcal{I})$, and choose  
$$ M>2\sup_{(x,\xi)} \vert a(x,\xi) \vert^{2}, $$ then $c(x,\xi)=(M-a(x,\xi)a(x,\xi)^*
 )^{1/2}\in S^{0,\mathcal{L}}_{\rho',\delta'}(M\times \mathcal{I}).$
Now, we have
$$  
c(x,D)^{*}c(x,D)=M-a^{*}(x,D)a(x,D)+r(x,D),
$$
where $r\in {S}^{-(\rho'-\delta')}_{{\rho',\delta'}}(M\times \mathcal{I}).$
Hence, $\Vert a(x,D) \Vert_{\mathscr{B}(L^2)}\leqslant M+\Vert r(x,D)\Vert_{\mathscr{B}(L^2)}. $
\end{proof}
 
 \begin{remark}
 For the $L^{p}$-$L^q$-boundedness of pseudo-differential operators in the setting of non-harmonic analysis we refer the reader to \cite{CardVishTokRuzI}.
 \end{remark}

\section{Global solvability  for evolution problems}\label{GST}

In this section we apply the G\r{a}rding inequality to some problems of PDEs, the global solvability of parabolic and hyperbolic type problems associated with the non-harmonic pseudo-differential calculus. More precisely, we study the existence and uniqueness of the solution of the Cauchy problem 
\begin{equation}
\label{PVI}
(\textnormal{PVI}): 
\begin{cases}\frac{\partial v}{\partial t}=K(t,x,D)v+f ,& 
\\v(0)=u_0, & \text{ } 
\end{cases}
\end{equation} 
where the initial data $u_0\in L^2(M),$ $K(t):=K(t,x,D)$ with a symbol in $S^m_{\rho,\delta}(M\times \mathcal{I}),$ $f\in  L^2([0,T]\times M) \simeq L^2([0,T],L^2(M)) ,$ $m>0,$ and  a suitable positivity condition is imposed on $K.$ 

We say that the problem \eqref{PVI} has a solution if there exists $v\in \mathscr{D}'((0,T)\times M)$ which satisfies the equation in  \eqref{PVI} with the initial condition $v(0)=u_0\in L^2(M)$ such that $v\in C^1([0,T],L^2(M))\bigcap C([0,T],\mathcal{H}^{m,L}(M)).$

In what follows, we assume that
\begin{equation*}
    \textnormal{Re}(K(t)):=\frac{1}{2}(K(t)+K(t)^*),\,\,0\leqslant t\leqslant T,\footnote{ This means that $A=K(t)$ is strongly $L$-elliptic.}
\end{equation*} 
is $L$-elliptic. Under such assumption we prove the existence and uniqueness of the solution $v\in C^1([0,T],L^2(M))\bigcap C([0,T],\mathcal{H}^{m,L}(M)).$ We start with the following energy estimate.

\begin{theorem}
\label{energyestimate}
Let $K(t)=K(t,x,D),$ $0\leqslant \delta<\rho\leqslant  1,$  be a  pseudo-differential operator of order $m>0$ with a symbol in $S^m_{\rho,\delta}(M\times \mathcal{I}).$ Assume that $\textnormal{Re}(K(t))$ is an $L$-elliptic operator, for every $t\in[0,T]$ with $T>0.$ If  
$$
v\in C^1([0,T], L^2(M) )  \bigcap C([0,T],\mathcal{H}^{m,L}(M))
$$ 
is a solution of the problem \eqref{PVI} then there exist $C, C'(T)>0$ such that
\begin{equation}
\Vert v(t)\Vert_{L^2(M)}\leqslant C \Vert u_0\Vert^2_{L^2(M)}+C'(T)\int\limits_{0}^T \Vert (\partial_{t}-K(\tau))v(\tau) \Vert^2_{L^2(M)}d\tau,  
\end{equation} 
holds for every $0\leqslant t\leqslant T.$ 

Moreover, we also have the estimate
\begin{equation}\label{Q*}
\Vert v(t)\Vert_{L^2(M)}\leqslant  C \Vert u_0\Vert^2_{L^2(M)} + C'(T) \int\limits_{0}^T \Vert (\partial_{t}-K(\tau)^{*})v(\tau) \Vert^2_{L^2(M)}d\tau.  
\end{equation}
\end{theorem}
\begin{proof} 
Let $v\in C^1([0,T], L^2(M) )  \cap C([0,T],\mathcal{H}^{m,L}(M)).$ Let us start by observing  that   $v\in  C([0,T],\mathcal{H}^{\frac{m}{2},L}(M))$ because of the embedding $\mathcal{H}^{m,L}\hookrightarrow \mathcal{H}^{\frac{m}{2},L}.$ This fact will be useful later because we will use the G\r{a}rding inequality applied to the operator $\textnormal{Re}(K(t)).$ So, 
$v\in \textnormal{Dom}(\partial_{\tau}-K(\tau))$ for every $0\leqslant \tau\leqslant T.$ In view of the embedding $\mathcal{H}^{m,L}\hookrightarrow L^2(M),$ we also have that  $v\in C([0,T], L^2(M) ).$  Let us define $f(\tau):=Q(\tau)v(\tau),$ $Q(\tau):=(\partial_{\tau}-K(\tau)),$ for every $0\leqslant \tau\leqslant T.$ Observe that

\begin{align*}
   \frac{d}{dt}\Vert v(t)\Vert^2_{L^2(M)}&= \frac{d}{dt}\left(v(t),v(t)\right)_{L^2(M)}\\&=\left(\frac{d v(t)}{dt},v(t)\right)_{L^2(M)}+\left(v(t),\frac{d v(t)}{dt}\right)_{L^2(M)}\\
   &=\left(K(t)v(t)+f(t),v(t)\right)_{L^2(M)}+\left(v(t),K(t)v(t)+f(t)\right)_{L^2(M)}\\
    &=\left((K(t)+K(t)^{*})v(t),v(t)\right)_{L^2(M)}+2\textnormal{Re}(f(t), v(t))_{L^2(M)}\\
     &=\textnormal{Re}(K(t)v(t),v(t))_{L^2(M)}+2\textnormal{Re}(f(t), v(t))_{L^2(M)}.
\end{align*}Now, from the  G\r{a}rding inequality, 
\begin{align}
    \textnormal{Re}(-K(t)v(t),v(t)) \geqslant C_1\Vert v(t)\Vert_{\mathcal{H}^{\frac{m}{2},L}(M)}-C_2\Vert v(t)\Vert_{L^2(M)}^2,
\end{align} 
and from the parallelogram law, we have
\begin{align*}
 2\textnormal{Re}(f(t), v(t))_{L^2(M)}&\leqslant 2\textnormal{Re}(f(t), v(t))_{L^2(M)}+\Vert f(t)\Vert_{L^2(M)}^2+\vert v(t)\vert_{L^2(M)}^2 \\
 &= \Vert f(t)+v(t)\Vert^2\leqslant \Vert f(t)+v(t)\Vert^2+\Vert f(t)-v(t)\Vert^2 \\
&= 2\Vert f(t)\Vert^2_{L^2(M)}+2\Vert v(t)\Vert^2_{L^2(M)}.
\end{align*}
Thus, we obtain
\begin{align*}
   & \frac{d}{dt}\Vert v(t)\Vert^2_{L^2(M)}\\
   &\leqslant 2\left(C_2\Vert v(t)\Vert_{L^2(M)}^2-C_1\Vert v(t)\Vert_{\mathcal{H}^{\frac{m}{2},L}(M)}\right)+2\Vert f(t)\Vert^2_{L^2(M)}+2\Vert v(t)\Vert^2_{L^2(M)}.
\end{align*}  
So, we have proved that
\begin{align*}
   \frac{d}{dt}\Vert v(t)\Vert^2_{L^2(M)}\lesssim  \Vert f(t)\Vert^2_{L^2(M)}+\Vert v(t)\Vert^2_{L^2(M)}.
\end{align*}
By using Gronwall's Lemma, we obtain the energy estimate
\begin{equation}
\label{GI-03}
\Vert v(t)\Vert^2_{L^2(M)}\leqslant  C\Vert u_0\Vert_{L^2(M)}^2+C'(T)\int\limits_{0}^T \Vert f(\tau) \Vert_{L^2(M)}^2d\tau,   
\end{equation}
for every $0\leqslant t\leqslant T,$ and $T>0.$ To finish the proof, we can change the calculations above with $v(T-\cdot)$ instead of $v(\cdot),$ $f(T-\cdot)$ instead of $f(\cdot)$ and $Q^{*}=-\partial_{t}-K(t)^{*},$ (or equivalently $Q=\partial_{t}-K(t)$ ) instead of $Q^{*}=-\partial_{t}+K(t)^{*}$ (or equivalently $Q=\partial_{t}-K(t)$) using that $\textnormal{Re}(K(T-t)^*)=\textnormal{Re}(K(T-t))$ to deduce that
\begin{align*}
&\Vert v(T-t) \Vert^2_{L^2(M)}\\
&\leqslant C\Vert u_0\Vert^2_{L^2(M)}+C'(T)\int\limits_{0}^{T} \Vert (-\partial_{t}+K(T-t)^{*})v(T-\tau) \Vert^2_{L^2(M)}d\tau \\
&=   C\Vert u_0\Vert^2_{L^2(M)}+C'(T)\int\limits_{0}^T \Vert (-\partial_{t}-K(t)^{*})v(s) \Vert^2_{L^2(M)}ds.
\end{align*}So, we conclude the proof.
\end{proof}

\begin{theorem}
Let $K(t)=K(t,x,D)\in S^m_{\rho,\delta}(M\times \mathcal{I}),$ $0\leqslant \delta<\rho\leqslant  1,$  be a  pseudo-differential operator of order $m>0,$ and let us assume that $\textnormal{Re}(K(t))$ is $L$-elliptic, for every $t\in[0,T]$ with $T>0.$ Let   $f\in L^2(M)$.  Then there exists a unique solution $v\in C^1([0,T], L^2(M) )  \bigcap C([0,T],\mathcal{H}^{m,L}(M))$ of the problem \eqref{PVI}. Moreover, $v$ satisfies the energy estimate
\begin{equation}
\Vert v(t)\Vert^2_{L^2(M)}\leqslant  \left(C\Vert u_0 \Vert^2_{L^2(M)}+C'\Vert f \Vert^2_{L^2([0,T],L^2(M))} \right),
\end{equation}for every $0\leqslant t\leqslant T.$
\end{theorem}
\begin{proof}
The energy estimate \eqref{GI-03} and the classical Picard iteration theorem implies the existence result. Now, in order to show the uniqueness of $v,$ let us assume that $u\in  C^1([0,T], L^2(M) )  \bigcap C([0,T],\mathcal{H}^{m,L}(M))$ is also a solution of the problem
\begin{equation*} \begin{cases}\frac{\partial u}{\partial t}=K(t,x,D)u+f ,& \text{ }u\in \mathscr{D}'((0,T)\times M),
\\u(0)=u_0 .& \text{ } \end{cases}
\end{equation*} Then $\omega:=v-u\in C^1([0,T], L^2(M) )  \bigcap C([0,T],\mathcal{H}^{m,L}(M))$ solves the problem
\begin{equation*} 
\begin{cases}
\frac{\partial \omega}{\partial t}=K(t,x,D)\omega,& \text{ }\omega\in \mathscr{D}'((0,T)\times M),
\\\omega(0)=0 .& \text{ } 
\end{cases}
\end{equation*}
From Theorem \ref{energyestimate} it follows that $\Vert \omega(t)\Vert_{L^2(M)}=0,$ for all $0\leqslant t\leqslant T.$ Hence, from the continuity in $t$ of the functions we have that $v(t,x)=u(t,x)$ for all $t\in [0,T]$ and a.e. $x\in M.$
\end{proof}

\bibliographystyle{amsplain}

\end{document}